\documentclass[12pt]{amsart}

\usepackage{graphicx,amssymb,latexsym,eepic}
\usepackage[dvips,margin=1.25in]{geometry}
\usepackage{verbatim}
\usepackage{multirow}

\usepackage[enableskew]{youngtab}
\usepackage{young}

\newcommand{\twop}{\mbox{$2'$}}
\newcommand{\threep}{\mbox{$3'$}}
\newcommand{\fourp}{\mbox{$4'$}}

\newcommand{\sixp}{\mbox{$6'$}}

\newcommand{\ninep}{\mbox{$9'$}}

\newcommand{\bprime}{\mbox{$b'$}}
\newcommand{\cprime}{\mbox{$c'$}}
\newcommand{\dprime}{\mbox{$d'$}}

\newcommand{\ten}{$10$}

\newcommand{\ycenter}{\Yvcentermath1}

\numberwithin{equation}{section}

\newtheorem{theorem}{Theorem}[section]
\newtheorem{proposition}[theorem]{Proposition}
\newtheorem{conjecture}[theorem]{Conjecture}
\newtheorem{corollary}[theorem]{Corollary}
\newtheorem{lemma}[theorem]{Lemma}

\theoremstyle{definition}

\newtheorem{remark}[theorem]{Remark}
\newtheorem{example}[theorem]{Example}
\newtheorem{definition}[theorem]{Definition}

\newcommand{\word}{\operatorname{read}}
\newcommand{\mread}{\operatorname{mread}}
\newcommand{\stan}{\operatorname{stan}}

\newcommand{\pl}{\mathbf{P}}
\newcommand{\shpl}{\mathbf{S}}
\newcommand{\sh}{\operatorname{shape}}
\newcommand{\seq}{\operatorname{\equiv}}

\newcommand{\rect}{\operatorname{rect}}
\newcommand{\plam}{\mathcal{P}_{\lambda}}

\newcommand{\mix}{\operatorname{mix}}
\newcommand{\rsk}{\operatorname{RSK}}
\newcommand{\sk}{\operatorname{SK}}
\newcommand{\sw}{\operatorname{SW}}

\newcommand{\pr}{P_{\rsk}}
\newcommand{\qr}{Q_{\rsk}}
\newcommand{\pmix}{P_{\mix}}
\newcommand{\qmix}{Q_{\mix}}
\newcommand{\psk}{P_{\sk}}
\newcommand{\qsk}{Q_{\sk}}
\newcommand{\pw}{P_{\sw}}
\newcommand{\qw}{Q_{\sw}}

\def\ZZ{{\mathbb Z}}

\def\RR{{\mathbb R}}
\def\QQ{{\mathbb Q}}

\def\uu{{\mathbf u}}
\def\xx{{\mathbf x}}

\title[The shifted plactic monoid]{The shifted plactic monoid}
\author{Luis Serrano}
\address{\hspace{-.3in} Department of Mathematics, University of Michigan,
Ann Arbor, MI 48109, USA}
\email{lserrano@umich.edu}
\date{\today}
\thanks{Partially supported by NSF grant DMS-0555880 
and by an NSERC Postgraduate Scholarship.
}
\keywords{plactic monoid, shifted tableau, mixed insertion, Schur $P$-function, shifted Littlewood-Richardson rule.}

\begin{document}
\ \vspace{.4in}

\maketitle

%\vspace{.3in}

\begin{abstract}
We introduce a shifted analog of the plactic monoid of Lascoux and Sch\"utzenberger, the \emph{shifted plactic monoid}. It can be defined in two different ways: via the \emph{shifted Knuth relations}, or using Haiman's mixed insertion.

Applications include: a new combinatorial derivation (and a new version of) the shifted Littlewood-Richardson Rule; similar results for the coefficients in the Schur expansion of a Schur $P$-function; a shifted counterpart of the Lascoux-Sch\"utzenberger theory of noncommutative Schur functions in plactic variables; a characterization of shifted tableau words; and more.
\end{abstract}

%\vspace{.2in}

\tableofcontents

\section*{Introduction}

\noindent \textbf{The (shifted) plactic monoid}.
The celebrated \emph{Robinson-Schensted-Knuth correspondence} \cite{Sch} is a bijection between words in a linearly ordered alphabet \linebreak $X = \{ 1<2<3<\cdots \}$ and pairs of \emph{Young tableaux} with entries in $X$. More precisely, each word corresponds to a pair consisting of a semistandard \emph{insertion tableau} and a standard \emph{recording tableau}. The words producing a given insertion tableau form a \emph{plactic class}. A.~Lascoux and M.~P.~Sch\"utzenberger \cite{LS} made a crucial observation based on a result by D.~E. Knuth \cite{Kn}: the plactic classes $[u]$ and $[v]$ of two words $u$ and $v$ uniquely determine the plactic class $[uv]$ of their concatenation. This gives the set of all plactic classes (equivalently, the set of all semistandard Young tableaux) the structure of a \emph{plactic monoid} $\pl = \pl (X)$. This monoid has important applications in representation theory and the theory of symmetric functions; see, e.g.,~\cite{LLT}.

The main goal of this paper is to construct and study a proper analog of the plactic monoid for (semistandard) \emph{shifted Young tableaux}, with similar properties and similar applications. The problem of developing such a theory was already posed more than 20 years ago by B.~Sagan \cite{Sa}. Shifted Young tableaux are certain fillings of a \emph{shifted shape}
 (a shifted Young diagram associated with a strict partition) with letters in an alphabet $X' = \{ 1' < 1 < 2' < 2 < \cdots \}$; see, e.g.,~\cite{SaUbiq}. M.~Haiman~\cite{Ha} defined the (shifted) \emph{mixed insertion correspondence}, a beautiful bijection between permutations and pairs of standard shifted Young tableaux; each pair consists of the mixed insertion tableau and the mixed recording tableau. Haiman's correspondence is easily generalized (see Section~\ref{mainsection}) to a bijection between words in the alphabet $X$ and pairs consisting of a semistandard shifted mixed insertion tableau and a standard shifted mixed recording tableau. (We emphasize that this bijection deals with words in the original alphabet $X$ rather than the extended alphabet $X'$.) We define a \emph{shifted plactic class} as the set of all words which have a given mixed insertion tableau. Thus, shifted plactic classes are in bijection with shifted semistandard Young tableaux. The following key property, analogous to that of Lascoux and Sch\"utzenberger's in the ordinary case, holds (Theorem \ref{monoidtheorem}): the shifted plactic class of the concatenation of two words $u$ and $v$ depends only on the shifted plactic classes of $u$ and $v$. Consequently, one can define the \emph{shifted plactic monoid} $\shpl = \shpl(X)$ in which the product is, again, given by concatenation. In analogy with the classical case, we obtain a presentation of $\shpl$ by the quartic \emph{shifted Knuth (or shifted plactic) relations}. So two words are shifted Knuth-equivalent if and only if they have the same mixed insertion tableau.
 
Sagan \cite{Sa} and Worley \cite{Wo} have introduced the \emph{Sagan-Worley correspondence}, another analog of Robinson-Schensted-Knuth correspondence for shifted tableaux. In the case of permutations, Haiman \cite{Ha} proved that the mixed insertion correspondence is dual to Sagan-Worley's. In Section \ref{wordsection}, we use a semistandard version of this duality to describe shifted plactic equivalence in yet another way, namely: two words $u$ and $v$ are shifted plactic equivalent if and only if the recording tableaux of their inverses (as biwords) are the same.

\medskip
\noindent \textbf{(Shifted) Plactic Schur functions}.
The \emph{plactic algebra}~$\QQ \pl$ is the semigroup algebra of the plactic monoid. The shape of a plactic class is the shape of the corresponding tableau. A \emph{plactic Schur function} $\mathcal{S}_{\lambda} \in \QQ \pl$ is the sum of all plactic classes of shape $\lambda$; it can be viewed as a noncommutative version of the ordinary Schur function~$s_{\lambda}$. This notion was used by Sch\"utzenberger~\cite{Sc} to obtain a proof of the Littlewood-Richardson rule along the following lines. It can be shown that the plactic Schur functions span the ring they generate. Furthermore, this ring is canonically isomorphic to the ordinary ring of symmetric functions: the isomorphism simply sends each Schur function~$s_{\lambda}$ to its plactic counterpart~$\mathcal{S}_{\lambda}$. It follows that the Littlewood-Richardson coefficient~$c^{\lambda}_{\mu, \nu}$ is equal to the coefficient of a fixed plactic class~$T_{\lambda}$ of shape~$\lambda$ in the product of plactic Schur functions $\mathcal{S}_{\mu} \mathcal{S}_{\nu}$. In other words, $c^{\lambda}_{\mu, \nu}$ is equal to the number of pairs $(T_{\mu}, T_{\nu})$ of plactic classes of shapes $\mu$ and $\nu$ such that $T_{\mu} T_{\nu} = T_{\lambda}$. 

We develop a shifted counterpart of this classical theory. The \emph{shifted plactic algebra} $\QQ \shpl$ is the semigroup algebra of the shifted plactic monoid, and a (shifted) \emph{plactic Schur $P$-function} $\mathcal{P}_{\lambda} \in \QQ \shpl$ is the sum of all shifted plactic classes of a given shifted shape. We prove that the plactic Schur $P$-functions span the ring they generate, and this ring is canonically isomorphic to the ring spanned/generated by the ordinary Schur $P$-functions. Again, the isomorphism sends each Schur $P$-function~$P_{\lambda}$ to its plactic counterpart~$\mathcal{P}_{\lambda}$. This leads to a proof of the shifted Littlewood-Richardson rule (Corollary~\ref{slrrule}). Our version of the rule states that the coefficient $b^{\lambda}_{\mu, \nu}$ of $P_{\lambda}$ in the product $P_{\mu} P_{\nu}$ is equal to the number of pairs $(T_{\mu} , T_{\nu})$ of shifted plactic classes of shapes $\mu$ and $\nu$ such that $T_{\mu}T_{\nu} = T_{\lambda}$, where $T_{\lambda}$ is a fixed shifted plactic class of shape $\lambda$. The first version of the shifted Littlewood-Richardson rule was given by Stembridge \cite{St}. In Lemma \ref{slrrulesarethesame} we relate our rule to Stembridge's by a simple bijection.

It turns out that the shifted plactic relations are a ``relaxation" of the ordinary Knuth (= plactic) relations. More precisely, the tautological map $u \mapsto u$ that sends each word in the alphabet $X$ to itself descends to a monoid homomorphism $\shpl \rightarrow \pl$. By extending this map linearly, we obtain the following theorem (Corollary \ref{pschurintermsofschur}): For a shifted shape $\lambda$, the coefficient $g^{\lambda}_{\mu}$ of $s_{\mu}$ in the Schur expansion of $P_{\lambda}$ is equal to the number of shifted plactic classes of shifted shape $\lambda$ contained in a fixed plactic class of shape~$\mu$. A simple bijection (Theorem \ref{usstem}) recovers a theorem of Stembridge~\cite{St}: $g^{\lambda}_{\mu}$ is equal to the number of standard Young tableaux of shape $\mu$ which rectify to a fixed standard shifted Young tableau of shape $\lambda$.

%It turns out that the shifted plactic relations are a ``relaxation" of the ordinary Knuth (= plactic) relations. More precisely, the tautological map $u \mapsto u$ that sends each word in the alphabet $X$ to itself descends to a monoid homomorphism $\shpl \rightarrow \pl$. By extending this map linearly, we obtain the following theorem (Corollary \ref{pschurintermsofschur}): For a shifted shape $\theta$, the coefficient $g^{\theta}_{\mu}$ of $s_{\mu}$ in the Schur expansion of $P_{\theta}$ is equal to the number of shifted plactic classes of shifted shape $\theta$ contained in a fixed plactic class of shape~$\mu$. A simple bijection (Theorem \ref{usstem}) recovers a theorem of Stembridge~\cite{St}: $g^{\theta}_{\mu}$ is equal to the number of standard Young tableaux of shape $\mu$ which rectify to a fixed standard shifted Young tableau of shape $\theta$.
\medskip

\noindent \textbf{(Shifted) Tableau words}.
In the classical setting, an approach developed by Lascoux and his school begins with the plactic monoid as the original fundamental object, and identifies each tableau $T$ with a distinguished canonical representative of the corresponding plactic class, the \emph{reading word} $\word(T)$. This word is obtained by reading the rows of $T$ from left to right, starting from the bottom row and moving up. A word $w$ such that $w = \word(T)$ for some tableau $T$ is called a \emph{tableau word}. By construction, tableau words are characterized by the following property. Each of them is a concatenation $w = u_l u_{l-1} \cdots u_1$ of weakly increasing words $u_l, \ldots, u_1$, such that
\begin{itemize}
\item[(A)] for $1 \le i \le l-1$, the longest weakly increasing subword of~$u_{i+1} u_i$ is~$u_i$.
\end{itemize}
For a tableau word $w$, the lengths of the segments $u_i$ are precisely the row lengths of the Young tableau corresponding to $w$. 

We develop an analog of this approach in the shifted setting by taking the shifted plactic monoid as the fundamental object, and constructing a canonical representative for each shifted plactic class. Since shifted Young tableaux have primed entries while the words in their respective shifted plactic classes have not, the reading of a shifted Young tableau cannot be defined in as simple a manner as in the classical case. Instead, we define the \emph{mixed reading word} $\mread(T)$ of a shifted tableau $T$ as the unique word in the corresponding shifted plactic class that has a distinguished \emph{special recording tableau}. The latter notion is a shifted counterpart of P.~Edelman and C.~Greene's \emph{dual reading tableau} \cite{EG}.

A word $w$ such that $w = \mread(T)$ for some shifted Young tableau $T$ is called a \emph{shifted tableau word}. Such words have a characterizing property similar to~(A), with weakly increasing words replaced by \emph{hook words} (a hook word consists of a strictly decreasing segment followed by a weakly increasing one). In Theorem~\ref{tableauwordtheorem} and Proposition~\ref{ssdtproposition}, we prove that $w$ is a shifted tableau word if and only if it is a concatenation of hook words $u_l, \ldots, u_1$ such that
\begin{itemize}
\setcounter{enumi}{1}
\item[(B)] for $1 \le i \le l-1$, the longest hook subword of~$u_{i+1} u_i$ is~$u_i$.
\end{itemize}
For a shifted tableau word $w$, the lengths of the segments $u_i$ are precisely the row lengths of the shifted Young tableau corresponding to $w$.

\medskip
\noindent \textbf{Semistandard decomposition tableaux}.
The proofs of our main results make use of the following machinery. Building on the concept of \emph{standard decomposition tableaux} introduced by W.~Kra\'skiewicz \cite{Kr} and further developed by T.~K.~Lam \cite{TKL}, we define a (shifted) \emph{semistandard decomposition tableau} (SSDT) $R$ of shifted shape $\lambda$ as a filling of $\lambda$ by entries in $X$ such that the rows $u_1, u_2, \ldots, u_l$ of $R$ are hook words satisfying~(B). We define the \emph{reading word} of $R$ by $\word(R) = u_l u_{l-1} \cdots u_1$, that is, by reading the rows of $R$ from left to right, starting with the bottom row and moving up.

As a semistandard analog of Kra\'skiewicz's correspondence \cite{Kr}, we develop the \emph{SK~correspondence} (see Definition \ref{skinsertion}). This is a bijection between words in the alphabet $X$ and pairs of tableaux with entries in $X$. Every word corresponds to a pair consisting of an SSDT  called the \emph{SK insertion tableau} and a standard shifted Young tableau called the \emph{SK recording tableau}. We prove (Theorem \ref{samerecordingtableau}) that the mixed recording tableau and the SK recording tableau of a word $w$ are the same. Furthermore, we construct (see Theorem \ref{bijection}) a bijection $\Phi$ between SSDT and shifted Young tableaux of the same shape that preserves the reading word: $\word(R) = \mread(\Phi(R))$. In light of the conditions (A) and (B) above, one can see that the counterpart of an SSDT in the ordinary case is nothing but a semistandard Young tableau.

\medskip

\noindent \textbf{Outline}.
We state our results in sections \ref{mainsection}-\ref{moreresults}, relegating the proofs to section \ref{proofssection}. Section \ref{mainsection} contains the descriptions of the shifted plactic monoid, and its main applications. Section \ref{moreresults} gives a characterization of shifted tableau words, and a description of them in terms of semistandard decomposition tableaux. Section \ref{proofssection} contains the proofs of the main theorems.

\medskip

\noindent \textbf{Acknowledgements}
I am grateful to Sergey Fomin for suggesting the problem and for his comments on the earlier versions of the paper. I would also like to thank Marcelo Aguiar, Curtis Greene, Peter Hoffman, Tadeusz J\'ozefiak, Alain Lascoux, Thomas Lam, Cedric Lecouvey, Pavlo Pylyavskyy, Bruce Sagan, John Stembridge, and Alex Yong for helpful and inspiring conversations.

\section{Main results}\label{mainsection}

    \subsection{Preliminaries}\textbf{Plactic monoid, shifted Young tableaux and the mixed insertion}

We assume the reader's familiarity with the basic theory of the ordinary plactic monoid \cite{LLT}; the goal of the swift review given in the next paragraph is mainly to introduce notation.

For a word $w = w_1 w_2 \cdots w_n$ in the alphabet $X$, let $\pr(w)$ and $\qr(w)$ denote its Robinson-Schensted-Knuth insertion and recording tableaux. Two words $u$ and $v$ in the alphabet $X$ are \emph{plactic equivalent} if $\pr(u) = \pr(v)$. Knuth \cite{Kn} has proved that the latter holds if and only if $u$ and $v$ are equivalent modulo the \emph{plactic relations}
\begin{equation}\label{k1}
acb \sim cab \quad \mbox{for} \quad a \le b < c \quad \mbox{in $X$},
\end{equation}
\begin{equation}\label{k2}
bca \sim bac \quad \mbox{for} \quad a < b \le c \quad \mbox{in $X$}.
\end{equation}

These relations define the \emph{plactic  monoid} $\mathbf{P} = \mathbf{P}(X)$ of Lascoux and Sch\"utzenberger~\cite{LS}.

A \emph{plactic class} is an equivalence class under plactic equivalence. The plactic class of a word $u$ in the alphabet $X$ is denoted $ \langle u \rangle$. Thus, $\pl$ is the set of plactic classes where multiplication is given by $\langle u \rangle \langle v \rangle = \langle uv \rangle $. Equivalently, it is generated by the symbols in $X$ subject to relations (\ref{k1})--(\ref{k2}).

A \emph{strict partition} is a sequence $\lambda = (\lambda_1, \lambda_2, \ldots, \lambda_l) \in \ZZ^l$ such that $\lambda_1 > \lambda_2 > \cdots > \lambda_l > 0$. The \emph{shifted diagram}, or \emph{shifted shape} of $\lambda$ is an array of square cells in which the $i$-th row has $\lambda_i$ cells, and is shifted $i-1$ units to the right with respect to the top row.

Throughout this paper, we identify a shifted shape corresponding to a strict partition~$\lambda$ with $\lambda$ itself.

The \emph{size} of~$\lambda$ is~$| \lambda | = \lambda_1 + \lambda_2 + \cdots + \lambda_l$. We denote $\ell(\lambda) = l$, the number of rows.

To illustrate, the shifted shape~$\lambda = (5,3,2)$, with~$| \lambda | = 10$ and~$\ell(\lambda) = 3$, is shown below:
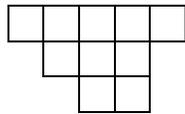
\[
 \young(~~~~~,:~~~,::~~).
\]

A \emph{skew shifted diagram} (or shape) $\lambda / \mu$ is obtained by removing a shifted shape $\mu$ from a larger shape $\lambda$ containing $\mu$.

A \emph{(semistandard) shifted Young tableaux} $T$ of shape $\lambda$ is a filling of a shifted shape~$\lambda$ with letters from the alphabet $X' = \{1' < 1 < 2' < 2 < \cdots \}$ such that:
\begin{itemize}
\item rows and columns of $T$ are weakly increasing;
\item each $k$ appears at most once in every column;
\item each $k'$ appears at most once in every row;
\item there are no primed entries on the main diagonal.
\end{itemize}
If $T$ is a filling of a shape $\lambda$, we write $\sh(T) = \lambda$.

A \emph{skew shifted Young tableau} is defined analogously.

The \emph{content} of a tableau $T$ is the vector $(a_1, a_2, \ldots)$, where~$a_i$ is the number of times the letters~$i$ and~$i'$ appear in $T$.

\begin{example}\label{shiftedtableauexample}
The shifted Young tableau
\[ \ycenter
T= \young(112{\threep}4,:455,::6{\ninep})
\]
has shape $\lambda = (5,3,2)$ and content $(2,1,1,2,2,1,0,0,1)$.
\end{example}

A tableau $T$ lf shape $\lambda$ is called \emph{standard} if it contains each of the entries $1, 2, \ldots, |\lambda|$ exactly once. In particular, standard shifted Young tableaux have no primed entries. Note that a standard shifted tableau has content $(1,1,\ldots,1)$.

M.~Haiman \cite{Ha} has introduced \emph{shifted mixed insertion}, a remarkable correspondence between permutations and pairs of shifted Young tableaux. Haiman's construction can be viewed as a shifted analog of the Robinson-Schensted-Knuth correspondence.

The following is a semistandard generalization of shifted mixed insertion, which we call \emph{semistandard shifted mixed insertion}. It is a correspondence between words in the alphabet $X$ and pairs of shifted Young tableaux, one of them semistandard and one standard. Throughout this paper we refer to semistandard shifted mixed insertion simply as \emph{mixed insertion}.

\begin{definition}[Mixed insertion]
Let $w = w_1 \ldots w_n$ be a word in the alphabet $X$. We recursively construct a sequence $(T_0,U_0), \ldots, (T_n, U_n) = (T,U)$ of tableaux, where $T_i$ is a shifted Young tableau, and $U_i$ is a standard shifted Young tableau, as follows. Set $(T_0, U_0) = (\emptyset, \emptyset)$. For $i = 1, \ldots, n$, insert $w_i$ into $T_{i-1}$ in the following manner:

Insert $w_i$ into the first row, bumping out the smallest element $a$ that is strictly greater than $w_i$ (in the order given by the alphabet $X'$).
\begin{enumerate}
\item if $a$ is not on the main diagonal, do as follows:
\begin{enumerate}
\item if $a$ is unprimed, then insert it in the next row, as explained above;
\item if $a$ is primed, insert it into the next column to the right, bumping out the smallest element that is strictly greater than $a$;
\end{enumerate}
\item if $a$ is on the main diagonal, then it must be unprimed. Prime it, and insert it into the next column to the right.
\end{enumerate}
The insertion process terminates once a letter is placed at the end of a row or column, bumping no new element. The resulting tableau is~$T_i$.

The shapes of $T_{i-1}$ and $T_i$ differ by one box. Add that box to $U_{i-1}$, and write $i$ into it to obtain $U_i$.

We call $T$ the \emph{mixed insertion tableau} and $U$ the \emph{mixed recording tableau}, and denote them $\pmix(w)$ and $\qmix(w)$, respectively.
\end{definition}

\begin{example}\label{mixedinsertionexample1}
The word $u = 3415961254$ has the following mixed insertion and recording tableau
\[
\ycenter
\pmix(u) = \young(112{\threep}4,:455,::6{\ninep}) \qquad \qmix(w) = \young(12459,:368,::7{\ten}).
\]
\end{example}

    \subsection{The shifted plactic monoid}

Theorem \ref{monoidtheorem} below is a shifted analog of the plactic relations (\ref{k1})--(\ref{k2}) \cite{Kn}. It can be considered a semistandard generalization of results by Haiman~\cite{Ha} and by Kra\'skiewicz~\cite{Kr}.

\begin{theorem}\label{monoidtheorem}
Two words have the same mixed insertion tableau if and only if they are equivalent modulo the following relations:
\begin{equation}\label{sk1}
\ycenter
abdc \seq adbc    \quad \mbox{for} \quad a \le b \le c < d \quad \mbox{in $X$};
\end{equation}
\begin{equation}\label{sk2}
\ycenter
acdb \seq acbd    \quad \mbox{for} \quad a \le b < c \le d \quad \mbox{in $X$};
\end{equation}
\begin{equation}\label{sk3}
\ycenter
dacb \seq adcb    \quad \mbox{for} \quad a \le b < c < d \quad \mbox{in $X$};
\end{equation}
\begin{equation}\label{sk4}
\ycenter
badc \seq bdac    \quad \mbox{for} \quad a < b \le c<d \quad \mbox{in $X$};
\end{equation}
\begin{equation}\label{sk5}
\ycenter
cbda \seq cdba    \quad \mbox{for} \quad a<b<c \le d \quad \mbox{in $X$};
\end{equation}
\begin{equation}\label{sk6}
\ycenter
dbca \seq bdca    \quad \mbox{for} \quad a<b \le c<d \quad \mbox{in $X$};
\end{equation}
\begin{equation}\label{sk7}
\ycenter
bcda \seq bcad    \quad \mbox{for} \quad a<b \le c \le d \quad \mbox{in $X$};
\end{equation}
\begin{equation}\label{sk8}
\ycenter
cadb \seq cdab    \quad \mbox{for} \quad a \le b<c \le d \quad \mbox{in $X$}.
\end{equation}
Consequently, the mixed insertion tableau of a concatenation of two words is uniquely determined by their mixed insertion tableaux.
\end{theorem}

\begin{remark}\label{placticremark}
As noted in \cite{Schutz}, the plactic relations (\ref{k1})--(\ref{k2}) can be understood in the following way. Let us call $w = w_1 \cdots w_n$ a \emph{line word} if
\[
w_1 > w_2 > \cdots > w_n
\]
or
\[
w_1 \le w_2 \le \cdots \le w_n.
\]
Line words are precisely those words $w$ for which the shape of $\pr(w)$ is a single row or a single column.

In this language, the relations (\ref{k1})--(\ref{k2}) can be stated as follows. Two 3-letter words $w$ and $w'$ in the alphabet $X$ are plactic equivalent if and only if:
\begin{itemize}
\item $w$ and $w'$ differ by an adjacent transposition, and
\item neither $w$ nor $w'$ is a line word.
\end{itemize}

The relations (\ref{sk1})--(\ref{sk8}) are called the \emph{shifted plactic relations}, and can be described in a similar way. Define a \emph{hook word} as a word $w = w_1 \cdots  w_l$ such that for some $1 \le k \le l$, we have
\begin{equation}\label{hookequation}
w_1 > w_2 > \cdots > w_k \le w_{k+1} \le \cdots \le w_l.
\end{equation}
It is easy to see that $w$ is a hook word if and only if $\pmix(w)$ consists of a single row.

Now, the shifted plactic relations (\ref{sk1})--(\ref{sk8}) are precisely the relations $w \seq w'$ in which:
\begin{itemize}
\item $w$ and $w'$ are plactic equivalent 4-letter words, and
\item neither $w$ nor $w'$ is a hook word.
\end{itemize}
\end{remark}

\begin{definition}
Two words $u$ and $v$ in the alphabet $X$ are \emph{shifted plactic equivalent} (denoted $u \seq v$) if they have the same mixed insertion tableau. By Theorem \ref{monoidtheorem}, $u$ and $v$ are shifted plactic equivalent if they can be transformed into each other using the shifted plactic relations~(\ref{sk1})--(\ref{sk8}).

A \emph{shifted plactic class} is an equivalence class under the relation $\seq$. The shifted plactic class containing a word $w$ is denoted by $[w]$. We can identify a shifted plactic class with the corresponding shifted Young tableau $T = \pmix(w)$, and write $[T] = [w]$.
\end{definition}
The Appendix at the end of the paper shows all kinds of shifted plactic classes of 4-letter words, and the corresponding tableau.
%For 4-letter words, Proposition \ref{refinementprop} is illustrated in the Appendix.% For some 5-letter words, in Example \ref{figureexample}.

\begin{example}\label{figureexample}
Figure \ref{shiftedplacticfigure} shows the shifted plactic classes of 5-letter words of content $(3,2)$, while Figure \ref{placticfigure} shows the plactic classes of the same. 
%Figure \ref{placticfigure} illustrates the shifted plactic classes of 5-letter words of content $(3,2)$, in the following way: the lines between two words indicate plactic equivalence, while the solid lines indicate shifted plactic equivalence. The inner boxes contain words in the same shifted plactic class, and the outer boxes contain words in the same plactic class. The tableau corresponding to each class is located beside each box.
\end{example}

\begin{figure}[htbp]
\begin{center}
\begin{picture}(300,145)(0,0)

%MAKE SIMPLE BOX WITH ONE WORD
%\put(0,30){\framebox(60,30)}
%\put(15,40){\makebox{11122}}

%boxes
\put(-15,120){\framebox(60,30)}
\put(55,120){\framebox(200,30)}
\put(265,120){\framebox(60,30)}
%words and lines between
\put(0,130){\makebox{11122}}
%\put(35,145){\line(1,0){30}}
\put(70,130){\makebox{11221}}
\put(105,135){\line(1,0){30}}
\put(140,130){\makebox{11212}}
\put(175,135){\line(1,0){30}}
\put(210,130){\makebox{12112}}
%\put(245,145){\line(1,0){30}}
\put(280,130){\makebox{21112}}
%tableau underneath
\put(-15,98){\young(11122)}
\put(140,85){\young(1112,:2)}
\put(265,98){\young(111{\twop}2)}

%boxes
\put(-15,30){\framebox(200,30)}
\put(195,30){\framebox(130,30)}
%words and lines between
\put(0,40){\makebox{22111}}
\put(35,45){\line(1,0){30}}
\put(70,40){\makebox{21211}}
\put(105,45){\line(1,0){30}}
\put(140,40){\makebox{21121}}
%\put(175,45){\line(1,0){30}}
\put(210,40){\makebox{12121}}
\put(245,45){\line(1,0){30}}
\put(280,40){\makebox{12211}}
%tableau underneath
\put(70,-5){\young(111{\twop},:2)}
\put(245,-5){\young(111,:22)}

\end{picture}
\end{center}
\caption{Shifted plactic classes of words of content $(3,2)$. Each box contains a shifted plactic class, with the edges representing shifted plactic relations; the corresponding shifted tableau is shown underneath.} 
\label{shiftedplacticfigure}
\end{figure}
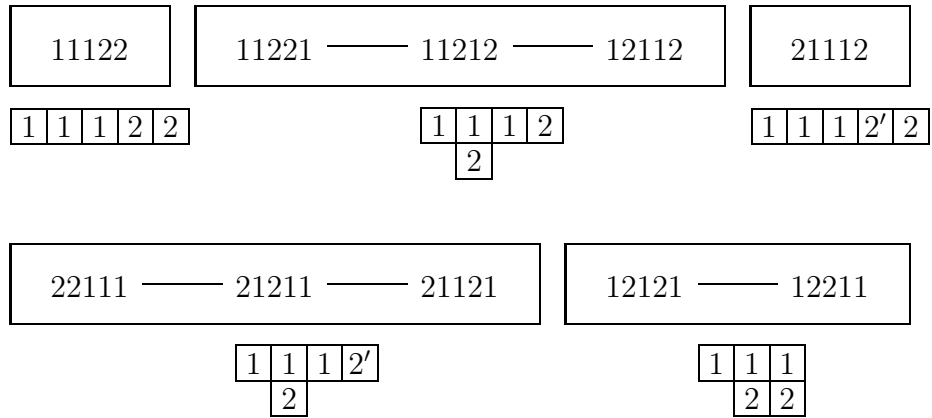

\begin{figure}[htbp]
\begin{center}
\begin{picture}(300,142.5)(0,0)

%MAKE SIMPLE BOX WITH ONE WORD
%\put(0,30){\framebox(60,30)}
%\put(15,40){\makebox{11122}}

%boxes
\put(-15,120){\framebox(60,30)}
\put(55,120){\framebox(270,30)}
%words and lines between
\put(0,130){\makebox{11122}}
%\put(35,145){\line(1,0){30}}
\put(70,130){\makebox{11221}}
\put(105,135){\line(1,0){30}}
\put(140,130){\makebox{11212}}
\put(175,135){\line(1,0){30}}
\put(210,130){\makebox{12112}}
\put(245,135){\line(1,0){30}}
\put(280,130){\makebox{21112}}
%tableau underneath
\put(-15,98){\young(11122)}
\put(175,85){\young(1112,2)}

%boxes
\put(-15,30){\framebox(340,30)}
%words and lines between
\put(0,40){\makebox{22111}}
\put(35,45){\line(1,0){30}}
\put(70,40){\makebox{21211}}
\put(105,45){\line(1,0){30}}
\put(140,40){\makebox{21121}}
\put(175,45){\line(1,0){30}}
\put(210,40){\makebox{12121}}
\put(245,45){\line(1,0){30}}
\put(280,40){\makebox{12211}}
%tableau underneath
\put(140,-5){\young(111,22)}

\end{picture}
\end{center}
\caption{Plactic classes of words of content $(3,2)$. Each box contains a plactic class, with the edges representing plactic relations; the corresponding tableau is shown underneath.} 
\label{placticfigure}
\end{figure}
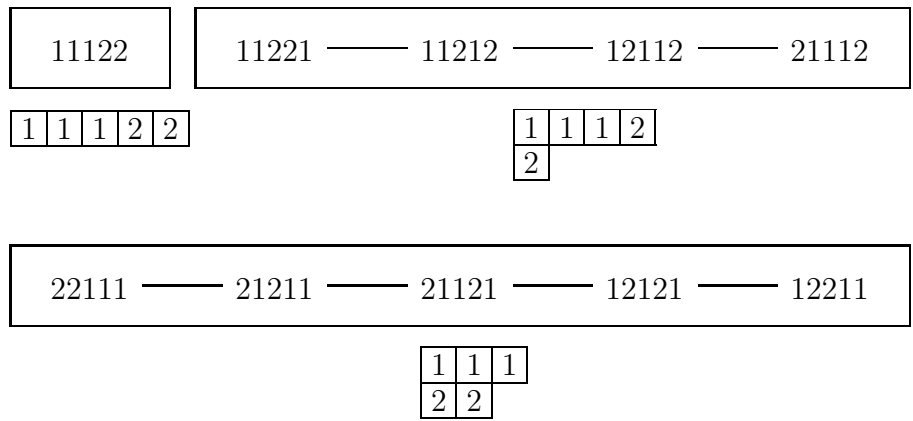

The following proposition can be verified by direct inspection.

\begin{proposition}\label{refinementprop}
Shifted plactic equivalence is a refinement of plactic equivalence. That is, each plactic class is a disjoint union of shifted plactic classes. To put it yet another way: if two words are shifted plactic equivalent, then they are plactic equivalent.
\end{proposition}

Proposition \ref{refinementprop} can be illustrated by comparing Figures \ref{shiftedplacticfigure} and \ref{placticfigure}.

\begin{definition}
The \emph{shifted plactic monoid} $\mathbf{S} = \mathbf{S}(X)$ is the set of shifted plactic classes with multiplication given by $[u][v] = [uv]$. (This multiplication is well defined by Theorem \ref{monoidtheorem}.) Equivalently, the monoid is generated by the symbols in $X$ subject to the relations~(\ref{sk1})--(\ref{sk8}).
\end{definition}

Alternatively, identifying each shifted plactic class with the corresponding shifted Young tableau, we obtain a notion of a (shifted plactic) associative product on the set of shifted tableaux.

The \emph{shape} of a shifted plactic class is defined as the shape of the corresponding shifted Young tableau.

The \emph{shifted plactic algebra} $\QQ \shpl$ is the semigroup algebra of the plactic monoid.

\begin{remark}
We normally consider $X$ as an infinite alphabet, but the totally analogous theory holds for any finite alphabet $X_n = \{ 1 < 2 < \cdots < n \}$.
\end{remark}

\subsection{Plactic Schur $P$-functions and their applications}
For a shifted Young tableau $T$ with content $(a_1, a_2, \ldots)$, we denote the corresponding monomial by
\[
x^T = x_1^{a_1} x_2^{a_2} \cdots.
\]

For each strict partition $\lambda$, the \emph{Schur $P$-function} is defined as the generating function for shifted Young tableaux of shape $\lambda$, namely
\[
P_{\lambda} = P_{\lambda} (x_1, x_2, \ldots) = \sum_{\sh (T) = \lambda} x^T.
\]
The \emph{Schur $Q$-function} is given by
\[
Q_{\lambda} = Q_{\lambda} (x_1, x_2, \ldots) = 2^{\ell(\lambda)} P_{\lambda},
\]
or equivalently, as the generating function for a different kind of shifted Young tableaux, namely those in which the elements in the main diagonal are allowed to be primed. For partitions of only one part, it is common to denote $P_{(k)}$ by $p_k$ and $Q_{(k)}$ by $q_k$.

The \emph{skew Schur $P$- and $Q$-functions}~$P_{\lambda / \mu}$ and~$Q_{\lambda / \mu} = 2^{\ell(\lambda) - \ell(\mu)} P_{\lambda / \mu}$ are defined similarly, for a skew shifted shape~$\lambda / \mu$. 

The following is an example of a Schur $P$-function in two variables:

\begin{example}\label{pschurexample}
For $\lambda = (3,1)$,
\[
\begin{array}{cccccccccccccc}
P_{\lambda}(x_1,x_2)    & =    & x_1^3 x_2    & +     & x_1^2 x_2^2    & + & x_1^2 x_2^2    & +    & x_1 x_2^3. \\[.1in] 
    &    & \young(111,:2)    &     & \young(11{\twop},:2)    &     & \young(112,:2)    &     & \young(1{\twop}2,:2)
\end{array}
\]
\end{example}

The Schur $P$- and $Q$-Schur functions form bases for an important subring $\Omega$ of the ring $\Lambda$ of symmetric functions.
    
The \emph{shifted Littlewood-Richardson coefficients}, $b^{\lambda}_{\mu, \nu}$ are of great importance in combinatorics, algebraic geometry, and representation theory. They appear in the expansion of the product of two Schur $P$-functions,
\[
P_{\mu} P_{\nu} = \sum_{\lambda} b^{\lambda}_{\mu, \nu} P_{\lambda}
\]
and also in the expansion of a skew Schur $Q$-function
\[
Q_{\lambda / \mu} = \sum_{\nu} b^{\lambda}_{\mu, \nu} Q_{\nu}.
\]
The latter can be rewritten as
\[
P_{\lambda / \mu} = \sum_{\nu} 2^{\ell(\mu) + \ell(\nu) - \ell(\lambda)} b^{\lambda}_{\mu, \nu} P_{\nu}.
\]

\begin{definition}
A \emph{shifted plactic Schur $P$-function} $\plam \in \QQ \shpl$ is defined as the sum of all shifted plactic classes of shape $\lambda$. More specifically, 
\[
\plam = \sum_{\sh(T) = \lambda} [T].
\]
\end{definition}

\begin{example}\label{placticschurexample}
We represent each shifted plactic class as $[w]$, for some representative~$w$, to obtain
\[
\begin{array}{cccccccccccccc}
\mathcal{P}_{(3,1)}    & =    & [1211]    & +     & [2211]    & + & [1212]    & +    & [2212]. \\[.1in] 
    &    & \young(111,:2)    &     & \young(11{\twop},:2)    &     & \young(112,:2)    &     & \young(1{\twop}2,:2)
\end{array}
\]
The reader can check that each word gets mixed inserted into the tableau underneath, making it a valid representative of its corresponding plactic class.
\end{example}
One can see that the $\plam$ are noncommutative analogs of the Schur $P$-functions. In the last example, $\mathcal{P}_{(3,1)}$ is a noncommutative analog of $P_{(3,1)}(x_1, x_2) = x_1^3 x_2 + 2 x_1^2 x_2^2 + x_1 x_2^3$.

\begin{theorem}\label{commutetheorem}
The map $P_{\lambda} \mapsto \plam$ extends to a canonical isomorphism between the algebra generated by the ordinary and shifted plactic Schur $P$-functions, respectively. As a result, the $\plam$ commute pairwise, span the ring they generate, and multiply according to the \emph{shifted Littlewood-Richardson rule}:
\begin{equation}\label{placticlreq}
\mathcal{P}_{\mu} \mathcal{P}_{\nu} = \sum_{\lambda} b^{\lambda}_{\mu, \nu} \plam .
\end{equation}
\end{theorem}

Sagan \cite{Sa} has extended the concept of jeu de taquin to shifted tableaux, and proved that, just as in the ordinary case, the result of applying a sequence of (shifted) jeu de taquin moves is independent from the order in which they are done. Throughout this paper we only apply shifted jeu de taquin to standard skew tableaux, for which the process is exactly as it is done in the ordinary case. For pairs of standard skew tableaux $T$ and $U$, we say that $T$ rectifies to $U$ if $U$ can be obtained from $T$ by a sequence of shifted jeu de taquin moves.

Our first application of Theorem \ref{commutetheorem} is a new proof (and a new version of) the shifted Littlewood-Richardson rule. Stembridge \cite{St} proved that the shifted Littlewood-Richardson number $b^{\lambda}_{\mu,\nu}$ is equal to the number of standard shifted Young skew tableaux of shape $\lambda / \mu$ which rectify to a fixed standard shifted Young tableau of shape~$\nu$.

By taking the coefficient of the shifted plactic class $[T]$ corresponding to a fixed tableau $T$ of shape $\lambda$ on both sides of (\ref{placticlreq}), one obtains the following:
\begin{corollary}[Shifted Littlewood-Richardson rule]\label{slrrule}
Fix a shifted plactic class $[T]$ of shape $\lambda$. The shifted Littlewood-Richardson coefficient $b^{\lambda}_{\mu, \nu}$ is equal to the number of pairs of shifted plactic classes $[U]$ and $[V]$ of shapes $\mu$ and $\nu$, respectively, such that $[U][V] = [T]$.
\end{corollary}

\begin{example}
Let us compute $b^{21}_{2,1}$. For this, we fix the shifted tableau word~$w = 132$ associated with the shifted Young tableau~$T = \ycenter \begin{tiny} \young(12,:3) \end{tiny}$. The only way to write~$w = uv$ where~$u$ and~$v$ are reading words of shapes~$(2)$ and~$(1)$, respectively, is with~$u = 13$ (associated to the tableau~$U = \ycenter \begin{tiny} \young(13) \end{tiny}$) and~$v = 2$ (associated to the tableau~$V = \ycenter \begin{tiny} \young(2) \end{tiny}$). We conclude that $b^{21}_{2,1} = 1$.
\end{example}

Corollary \ref{slrrule} can be restated in the language of words as follows. In Section \ref{wordsection} we introduce a canonical representative of the shifted plactic class $[T]$ corresponding to the tableau $T$. This representative is called the \emph{mixed reading word} of~$T$, and denoted by~$\mread(T)$. (See Definition \ref{mreaddef} for the precise details.) A word $w$ is called a \emph{shifted tableau word} if $w = \mread(T)$ for some shifted Young tableau $T$. The \emph{shape} of a shifted tableau word is, by definition, the shape of the corresponding tableau. With this terminology, the shifted Littlewood-Richardson rule can be restated as follows:

\begin{corollary}
Fix a shifted tableau word $w$ of shape $\lambda$. The shifted Littlewood-Richardson coefficient $b^{\lambda}_{\mu, \nu}$ is equal to the number of pairs of shifted tableau words $u, v$ of shapes $\mu, \nu$, respectively, such that $w \seq uv$.
\end{corollary}

The representatives we have picked in Example \ref{placticschurexample} are precisely the mixed reading words of the corresponding tableaux, as one can see in Example \ref{mrwexample}.

\begin{lemma}\label{slrrulesarethesame}
Fix a shifted tableau word $w$ of shape $\lambda$ and a standard shifted tableau $Q$ of shape $\nu$. The number of pairs of shifted tableau words $u$, $v$ of shapes $\mu$ and $\nu$, respectively, such that $uv = w$ is equal to the number of standard shifted skew tableaux $R$ of shape $\lambda / \mu$ which rectify to $Q$.
\end{lemma}

As a corollary, we recover the original result of Stembridge \cite{St}.

\begin{corollary}[\cite{St}]\label{lrsrule}
Fix a standard shifted tableau $Q$ of shape $\nu$. The shifted Littlewood-Richardson coefficient $b^{\lambda}_{\mu, \nu}$ is equal to the number of standard shifted skew tableaux of shape $\lambda / \mu$ which rectify to $Q$.
\end{corollary}

The second application is a new proof (and a new version of) the \emph{Schur expansion of a Schur $P$-function}. Stembridge \cite{St} has found a combinatorial interpretation for the coefficients $g^{\lambda}_{\mu}$ appearing in the sum
\begin{equation}\label{schurpschur}
P_{\lambda} = \sum_{\mu} g^{\lambda}_{\mu} s_{\mu}.
\end{equation}
Below we give a different description of the numbers $g^{\lambda}_{\mu}$ in terms of shifted plactic classes.

By Proposition \ref{refinementprop}, any two shifted plactic equivalent words are plactic equivalent; in other words, relations (\ref{sk1})--(\ref{sk8}) follow from (\ref{k1})--(\ref{k2}). This yields the natural projection
\[
\pi: \shpl \rightarrow \pl
\]
which maps the shifted plactic class~$[u]$ to the plactic class~$\langle u \rangle$.

We next consider the image of a plactic Schur $P$-function under $\pi$.

\begin{theorem}\label{pschurintermsofschur}
The plactic Schur $P$-function~$\mathcal{P}_{\lambda}$ is sent by $\pi$ to a sum of plactic Schur functions~$\mathcal{S}_{\mu}$. Specifically (cf. \ref{schurpschur}),
\[
\pi(\mathcal{P}_{\lambda}) = \sum_{\mu} g^{\lambda}_{\mu} \mathcal{S}_{\mu}.
\]
\end{theorem}
Since the span of the $\mathcal{P}_{\lambda}$ and the span of the $\mathcal{S}_{\lambda}$ are isomorphic to $\Omega$ and $\Lambda$, respectively, the following statement holds.
\begin{corollary}
The coefficient $g^{\lambda}_{\mu}$ is equal to the number of shifted plactic classes~$[u]$ of shifted shape $\lambda$ such that~$\pi([u]) = \langle v \rangle $ for some fixed plactic class~$\langle v \rangle $ of shape $\mu$.
\end{corollary}

\begin{example}
Let $\mu$ be the ordinary shape $(3,1)$, and $\lambda$ be the shifted shape $(3,1)$. Let us compute $g^{\lambda}_{\mu}$, the coefficient of $s_{\mu}$ in $P_{\lambda}$. For this, we fix $\langle u \rangle = \langle 2134 \rangle$, the plactic class corresponding to the Young tableau $U = \ycenter \begin{tiny} \young(134,2) \end{tiny}$. Note that the words in $\langle u \rangle$ are $2134, 2314,$ and $2341$. These get split into two shifted plactic classes, namely $[2134]$ corresponding to the shifted Young tableau $\ycenter \begin{tiny} \young(1{\twop}34) \end{tiny}$, and $[2314] = [2341]$ corresponding to the shifted Young tableau $\ycenter \begin{tiny} \young(1{\twop}4,:3) \end{tiny}$. Since only one of these shifted plactic classes has shape $\lambda$, namely $[2314]$, we conclude that $g^{\lambda}_{\mu} = 1$.
\end{example}

\begin{theorem}\label{usstem}
Let $\lambda$ be a shifted shape and $U_{\lambda}$ a fixed standard shifted tableau of shape $\lambda$. Fix a plactic class $\langle T_{\mu} \rangle$ of (ordinary) shape $\mu$. Then the number of shifted plactic classes $[T_{\lambda}]$ of shape $\lambda$ for which $\pi([T_{\lambda}]) = \langle T_{\mu} \rangle$ is equal to the number of standard Young tableaux of shape $\mu$ which rectify to $U_{\lambda}$.
\end{theorem}

As a corollary, we recover another result of Stembridge \cite{St}.

\begin{corollary}\label{stembridgerule}
The coefficient $g^{\lambda}_{\mu}$ is equal to the number of standard Young tableaux of shape~$\mu$ which rectify to a fixed standard shifted Young tableau of shape~$\lambda$.
\end{corollary}

\subsection{Noncommutative Schur $P$-functions and box-adding operators}

Fomin and Greene \cite{FG} have developed a theory of noncommutative Schur functions, and used it to obtain generalizations of the Littlewood-Richardson rule. A similar approach can be applied to the shifted case.

\begin{definition}[Noncommutative Schur $P$-function]\label{noncommdef}
Let $u_1, u_2, \ldots$ be a finite or infinite sequence of elements of some associative algebra $\mathcal{A}$ over $\QQ$. (We will always assume that these elements satisfy the shifted plactic relations.) For a shifted shape~$\lambda$, define
$$P_{\lambda}(\uu) = P_{\lambda}(u_1, u_2, \ldots) = \sum_T u^T,$$
where $T$ runs over all shifted Young tableaux, and the monomial $u^T$ is determined by $\mread(T)$, or by any other representative of the shifted plactic class corresponding to $T$. (In the infinite case, $P_{\lambda}(\mathbf{u})$ is an element of the appropriate \emph{completion} of the algebra $\mathcal{A}$.)
\end{definition}

Theorem \ref{commutetheorem} implies the following result.

\begin{corollary}\label{pcommute}
Assume that the elements $u_1, u_2, \ldots$ of some associative algebra satisfy the shifted plactic relations (\ref{sk1})--(\ref{sk8}) (the element $u_i$ represents the letter $i$ in the alphabet $X$).
Then the $P_{\lambda}(\uu)$ introduced in Definition \ref{noncommdef} commute pairwise, and satisfy the shifted Littlewood-Richardson rule:
$$P_{\mu}(\uu) P_{\nu}(\uu) = \sum_{\lambda} b^{\lambda}_{\mu, \nu} P_{\lambda}(\uu).$$
\end{corollary}

\begin{corollary}[Noncommutative Cauchy identity]\label{nci}
Let $u_1, u_2, \ldots, u_n$ be as in Corollary \ref{pcommute}, and let $x_1, x_2, \ldots, x_m$ be a family of commuting indeterminates, also commuting with each of the $u_j$. Then
\begin{equation}\label{shiftedcauchy}
\sum_{\lambda} P_{\lambda}(\uu) Q_{\lambda}(\xx) = \prod_{i = 1}^m \left( \prod_{j = n}^1 (1 + x_i u_j) \prod_{j = 1}^{n} (1 - x_i u_j)^{-1} \right).
\end{equation}
The analogous statement also holds when the $x_i$, or $u_j$, or both, are an infinite family.
\end{corollary}
\begin{proof}
We have
\begin{eqnarray*}
\prod_{i=1}^m \left( \prod_{j = \infty}^1 (1 + x_i u_j) \prod_{j = 1}^{\infty} (1 - x_i u_j)^{-1} \right)     & = & \prod_{i=1}^m \left( \sum_{k \ge 0} x_i^k \sum_{a_1 > \ldots > a_i \le a_{i+1} \le a_n} u_{a_1} \ldots u_{a_k} \right) \\
    & = & \prod_{i =1}^m \sum_{k \ge 0} x_i^k q_k(\uu).
\end{eqnarray*}

The last step follows from the classical shifted Cauchy identity (see e.g. \cite[Corollary 8.3]{Sa}) together with Corollary \ref{pcommute}, since now the $x_i$ and the $q_k(\mathbf{u})$ form a commuting family of indeterminates. Note that for this reason, the ordering in the interior products in \ref{shiftedcauchy} is as indicated, whereas the ordering in the outer products does not matter.
%\[
%\prod_{i =1}^m \sum_{k \ge 0} x_i^k p_k(\uu) = \sum_{\lambda} P_{\lambda}(\uu) Q_{\lambda}(\xx).
%\]
\end{proof}

\begin{definition}[Partial maps, cf.\ \cite{FG}]
Let $\mathbf{Y}$ be a finite or countable set, and let $\RR \mathbf{Y}$
be the vector space formally spanned over~$\RR$ by the elements
of~$\mathbf{Y}$. 
A linear map $u\in\operatorname{End}(\RR \mathbf{Y})$ is called 
a \emph{partial map} in~$\mathbf{Y}$ if the image $u(p)$ of each
element $p \in \mathbf{Y}$ is
either another element of $\mathbf{Y}$ or zero. 
\end{definition}

\begin{definition}[Generalized skew Schur $Q$-functions, cf. \cite{FG}]
Let $u_1, u_2, \ldots, u_n$ be partial maps in~$\mathbf{Y}$.
For any $g, h \in \mathbf{Y}$, define
\begin{equation}\label{genskew}
G_{h/g}(x_1, \ldots, x_m) = \left\langle \prod_{i=1}^{m} \left(
\prod_{j=n}^1 (1+x_i u_j) \prod_{j = 1}^{n} (1-x_i u_j)^{-1} \right)
g  ,  h  \right\rangle,
\end{equation}
where the variables $x_i$ commute with each other and with the~$u_j$, and the
noncommuting factors of the double product are multiplied in the
specified order. Here, $\langle* , *\rangle$ denotes the inner product on $\RR
\mathbf{Y}$ for which the elements of $\mathbf{Y}$ form an orthonormal
basis. By the argument at the end of the proof of Corollary \ref{nci}, the order of the factors in the outer product doesn't matter, which implies that the $G_{h/g}$ are (ordinary) symmetric polynomials in $x_1, \ldots, x_m$.
\end{definition}

\begin{theorem}[\emph{Generalized shifted Littlewood-Richardson Rule}]\label{glr}
Let the $u_i$ be partial maps in~$\mathbf{Y}$ satisfying the shifted plactic relations (\ref{sk1})--(\ref{sk8}). Then for any $g, h \in \mathbf{Y}$, the polynomial $G_{h/g}$ defined by (\ref{genskew}) is a nonnegative integer combination of Schur $Q$-functions. More specifically, 
\[
G_{h/g}(x_1, \ldots, x_m)
=\sum b_{g, \lambda}^h \, Q_{\lambda}(x_1, \ldots, x_m) ,
\]
where $b_{g, \lambda}^h$
is equal to the number of shifted semistandard Young tableaux~$T$ of
shape~$\lambda$ such that $u^T g = h$. 
\end{theorem}
\begin{proof}
By the noncommutative Cauchy identity (Corollary \ref{nci}),
\[
G_{h / g}(\xx) = \left\langle \sum_{\lambda} P_{\lambda}(\uu) Q_{\lambda}(\xx) g , h \right\rangle = \sum_{\lambda} \langle P_{\lambda}(\uu) g , h \rangle Q_{\lambda}(\xx).
\]
Consequently,
\[
b_{g \nu}^{h} = \langle P_{\lambda}(\uu) g , h \rangle,
\]
which is precisely the number of shifted semistandard Young tableaux~$T$ of
shape~$\lambda$ such that $u^T g = h$. 
\end{proof}

As an application of this theory, we obtain another version of the shifted Littlewood-Richardson Rule.

\begin{definition}[cf. \cite{FG}]
The \emph{diagonal box-adding operators} $u_j$ act on shifted shapes according to the following rule:
\[
u_j (\lambda)=     
\begin{cases}
\lambda \cup  \{ \text{box on the $j$-th diagonal}\} & \text{if this gives a valid shape}; \\
0 &  \text{otherwise.}
\end{cases}
\]
Here the diagonals are numbered from left to right, starting with $j=1$ for the main diagonal.
\end{definition}

\begin{example}
We have
\[
\ycenter u_2 \left( \young(~~~,:~) \right) = \young(~~~,:~~) \qquad \text{and} \qquad u_1 \left( \young(~~~,:~) \right) = 0.
\]
\end{example}

The maps $u_i$ are partial maps on the vector space formally spanned by the shifted shapes, as the image of each $u_i$ is either a shifted shape or $0$.

The product
\[
\mathcal{B}(x) = \prod_{j = n}^1(1+xu_j) \prod_{j = 1}^n(1-xu_j)^{-1}
\]
can be viewed as an operator that adds a (possibly disconnected) border strip to a fixed shifted shape (the first product will add a horizontal strip, and the second one a vertical strip), each time introducing a power of $x$ that is determined by the length of the strip. Setting $g = \mu$ and $h = \lambda$, we get
\[
\prod_{i \ge 0} \left( \prod_{j = n}^1(1+x_iu_j) \prod_{j = 1}^n(1-x_iu_j)^{-1}\right) \mu = \prod_{i \ge 0} \mathcal{B}(x_i) \mu  = \sum_{T} x^T \lambda,
\]
where the sum is over all semistandard skew shifted Young tableaux $T$ of shape $\lambda / \mu$. Therefore
\[
G_{\lambda / \mu} (\xx) = \left\langle \prod_{i \ge 1} \mathcal{B}(x_i) \mu, \lambda \right\rangle = \sum_{T} x^T = Q_{\lambda / \mu}(\xx),
\]
the skew Schur $Q$-function.

One can see that the box-adding operators $u_i$ satisfy the \emph{nil-Temperley-Lieb relations of type B} (cf. \cite[Example 2.4]{FG} \cite{FK}):
\begin{eqnarray*}
u_i u_j = u_j u_i \qquad & & |i-j| \ge 2,\\
u_i^2 = 0 \qquad & & i \ge 1,\\
u_i u_{i+1} u_i = 0 \qquad &  & i \ge 2,\\
u_{i+1} u_i u_{i+1} = 0 \qquad &  & i \ge 1,
\end{eqnarray*}
which implies that they also satisfy (\ref{sk1})--(\ref{sk8}). Consequently, we can use Theorem~\ref{glr} to obtain an expansion of the skew Schur $Q$-functions in terms of Schur $Q$-functions. This leads to the following new version of the shifted Littlewood-Richardson rule.

\begin{corollary}\label{lrcoef}
The shifted Littlewood-Richardson number $b_{\mu, \nu}^{\lambda}$ is equal to the number of shifted semistandard tableaux~$T$ of shape~$\nu$ such that $u^T(\mu) = \lambda$, where, as before, $u^T$ is the noncommutative monomial in $u_1, u_2, \ldots$ defined by any representative of the shifted plactic class associated with $T$, and each $u_i$ is interpreted as a diagonal box-adding operator.
\end{corollary}

The following is a direct connection between the versions of the shifted Littlewood-Richardson rule given in Corollary \ref{lrcoef} and in Corollary~\ref{lrsrule}.

Let $T$ be a tableau of shape $\nu$ with the property that $u^T \mu = \lambda$. We wish to match it to a standard filling of a shape $\lambda / \mu$ that rectifies to a fixed standard tableau of shape $\nu$. This filling will be obtained as follows: Note that $\mread(T)$ adds boxes to a shape~$\mu$ one by one, until a shape $\lambda$ is obtained. If we add a label $i$ to the $i$-th box added in this procedure, we  obtain a filling of a shape $\lambda / \mu$ that will rectify to the tableau of shape~$\nu$ in which the first box in the $i$-th row will contain the number $\lambda_1 + \ldots + \lambda_{i-1}+1$, and the numbers will be increasing by one in every row.

\begin{example}
For the computation the coefficient $b^{542}_{31, 43}$, one has the following standard shifted Young tableau of shape $(542) / (31)$ which rectifies to a fixed tableau (say, the one below) of shape~$(43)$, as follows:
\[
\ycenter
\young(~~~14,:~235,::67) \longrightarrow \young(1234,:567).
\]
If one records the diagonal in which the numbers $7,6,\ldots, 1$ are located on the skew shifted tableau at the left, one gets the sequence $w = 2145324$. Note that $w$ is the mixed reading word of the shifted semistandard Young tableau
\[
\young(1{\twop}24,:3{\fourp}5)
%\young(4213,:103),
\]
which means that its corresponding monomial $u_2u_1u_4u_5u_3u_2u_4$ comes from $P_{(4,3)}(\uu)$. Note that this monomial, when applied to a shape~$(3,1)$ gives a shape~$(5,3,2)$, exactly as the rule states.
\end{example}

\section{More results}\label{moreresults}

\subsection{Shifted tableau words}\label{wordsection}

In this section, we systematically develop the theory of the shifted plactic monoid by defining a canonical representative of each shifted plactic class and stating its characterizing properties.

Recall that the mixed insertion correspondence associates each word with its insertion and recording tableau. Associating each word in a shifted plactic class $[T]$ with its recording tableau gives a natural correspondence between representatives of $[T]$ and standard tableaux of the same shape as $T$. Thus, by constructing a canonical standard shifted Young tableau of each shape, we specify a canonical representative of each shifted plactic class.

A skew shifted tableau $T$ is a \emph{vee} if its shape is a (possibly disconnected) border strip, and the entries $i, i+1, \ldots, k$ appear in $T$ in the following manner. There exists some $i \le j \le k$ such that:
\begin{itemize}
\item the entries $i, i+1, \ldots, j$ form a vertical strip;
\item these entries are increasing down the vertical strip;
\item the entries $j, j+1, \ldots, k$ form a horizontal strip;
\item these entries are increasing from left to right;
\item each box in the vertical strip is left of those boxes in the horizontal strip that are on the same row.
\end{itemize}
The size of a vee is the number $k$ of entries. A vee is \emph{connected} if its corresponding skew shape is.

\begin{example}
A vee of shape $(5,3,2) / (3,2)$ and size $5$.
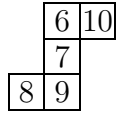
\[
\young(:6{\ten},:7,89)
\]
\end{example}

\begin{definition}
A standard shifted Young tableau of shape $\lambda = (\lambda_1, \lambda_2, \ldots, \lambda_l)$ is a \emph{special recording tableau} if it is standard, and for every $i$ such that $1 \le i \le l$, the entries $\lambda_l + \ldots + \lambda_{i+1}+1, \lambda_l + \ldots + \lambda_{i+1}+2, \ldots, \lambda_l + \ldots +\lambda_{i+1} + \lambda_i$ form a connected vee.
\end{definition}

Note that for every shape, the special recording tableau is unique.

This concept is a shifted analog of the \emph{dual reading tableau} defined by P.~Edelman and C.~Greene \cite{EG}.

\begin{example}\label{specialrecordingtableauexample}
The steps for building a special recording tableau of shape $(5,3,2)$, by adding connected vees formed by the following sets of numbers: $\{1,2\}, \{3,4,5\}$, and $\{6,7,8,9,10\}$:
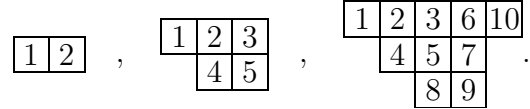
\[
\ycenter \young(12) \quad , \quad \young(123,:45) \quad , \quad \young(1236{\ten},:457,::89).
\]
\end{example}

\begin{definition}\label{mreaddef}
The \emph{mixed reading word} of a shifted Young tableau $T$ is the word corresponding to the pair $(T, U)$ under the mixed insertion correspondence, where $U$ is the shifted dual reading tableau of the same shape as $T$. The mixed reading word of $T$ is denoted $\mread(T)$.
\end{definition}

\begin{example}
We find $\mread(T)$ for the tableau $T$ in Example \ref{shiftedtableauexample}. Recall that the special recording tableau of shape $(5,3,2)$ is $U$, given in Example \ref{specialrecordingtableauexample}. 

To obtain the last letter of the mixed reading word, one first removes (using the mixed insertion algorithm backwards) the element in $T$ corresponding to the largest entry in $U$, namely the $10$. In $T$, this is precisely the~$4$ in the top right corner. Thus, the last letter of $\mread(T)$ is $4$. One continues in this fashion, to obtain $\mread(T) = 3451196524$.
\end{example}

\begin{example}\label{mrwexample}
The mixed reading words of all tableaux of shape $\lambda = (3,1)$ in the alphabet~$\{ 1,2 \}$:
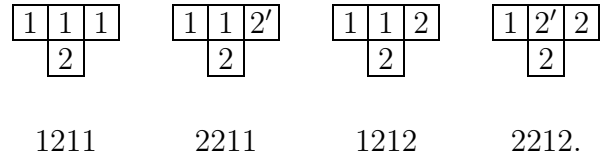
\[ \begin{array}{cccccccccccccc}
\young(111,:2)    &     & \young(11{\twop},:2)    &     & \young(112,:2)    &     & \young(1{\twop}2,:2) \\[.2in] 
1211    &      & 2211    &  & 1212    &     & 2212.
\end{array}
\]
\end{example}

\begin{definition}
A word $w$ in the alphabet $X$ is a \emph{shifted tableau word} if there exists a shifted Young tableau $T$ such that $w = \mread(T)$. The shape of a shifted tableau word is given by the shape of the corresponding tableau.
\end{definition}

\begin{theorem}
Every shifted plactic class $[T]$ contains exactly one shifted tableau word, namely $\mread(T)$.
\end{theorem}

\begin{proof}
The fact that $\pmix(\mread(T)) = T$ is direct from the definition. The uniqueness follows from the fact that each shifted plactic class corresponds to a unique shifted Young tableau.
\end{proof}

We proceed to characterize shifted tableau words by certain properties.

Recall that a hook word is a word $w = w_1 \cdots  w_l$ such that for some $1 \le k \le l$, the inequalities (\ref{hookequation}) hold. It is formed by the \emph{decreasing part} $w \downarrow   = w_1 \cdots w_k$, and the \emph{increasing part} $w \uparrow = w_{k+1} \cdots w_l$. Note that the decreasing part of a hook word is always nonempty.

\begin{theorem}\label{tableauwordtheorem}
A word $w$ is a shifted tableau word if and only if it is of the form $w = u_l u_{l-1} \cdots u_1$, and:
\begin{enumerate}
\item each $u_i$ is a hook word,
\item $u_i$ is a hook subword of maximum length in $u_l u_{l-1} \cdots u_i$, for $1 \le i \le l-1$.
\end{enumerate}
Furthermore, for $1 \le i \le l$, the shape of the tableau $\pmix(u_l\cdots u_{i+1} u_i)$ is $(\lambda_i, \lambda_{i+1}, \ldots, \lambda_i)$, where $|u_i| = \lambda_i$. In particular, $\sh(\pmix(u)) = \lambda$.
\end{theorem}

\begin{proposition}\label{ssdtproposition}
An equivalent definition of a shifted tableau word is obtained when one replaces condition \emph{(2)} in Theorem \ref{tableauwordtheorem} by the following:
\begin{enumerate}
\item[($2'$)] $u_i$ is a hook subword of maximum length in $u_{i+1} u_i$, for $1 \le i \le l-1$.
\end{enumerate}
\end{proposition}

We have been unable to construct an analog of jeu de taquin for skew semistandard shifted tableaux that is consistent with the shifted plactic relations. However, we have found a way to define analogs of the mixed reading word and the rectification $\rect(T)$ of such a tableau $T$.

Consider the alphabets $-X = \{ \cdots < -3 < -2 < -1 \}$ and $-X' = \{ \cdots -2 < -2' < -1 < -1'\}$, where by convention, any letter of $-X \, \, (-X')$ is smaller than any letter in $X \, \,(X')$. The mixed reading word of a skew shifted Young tableau $T$ of shape $\lambda / \mu$ is defined as follows: Fill the shape $\mu$ with any shifted Young tableau with letters in $-X'$. In this way, one obtains a shifted Young tableau $U$ with letters in $X' \cup -X'$. Note that $\mread(U)$ has letters in the alphabet $X \cup -X$. Let $\mread(T)$ be the restriction of $\mread(U)$ to the alphabet $X$, and let $\rect(T)$, the \emph{rectification} of $T$ be $\pmix(\mread(T))$. The following lemma confirms that this reading is well defined.

\begin{lemma}\label{samereading}
Let $T$ be a shifted Young skew tableau of shape $\lambda / \mu$. Both $\mread(T)$ and $\rect(T)$ are independent of the filling of the shape $\mu$ with letters in $-X$.
\end{lemma}

\begin{proof}
Let $a$ be a letter in $X \cup -X$, and $U$ a shifted Young tableau with letters in $X' \cup -X'$. The mixed insertion of $a$ into $U$ gives rise to a sequence of letters $a = a_1, a_2, \ldots, a_k = b$, each getting bumped at every stage, until $b$ gets added to a new row or column. By definition, $a_1 < a_2 < \cdots < a_k$. Thus, when one removes $b$ using the inverse process, the letters bumped are $b = a_k > \cdots > a_1 = a$. Thus, if $a \in -X'$, there is some $1 \le i \le k$ such that $a_i, a_{i-1}, \ldots, a_1 \in -X'$. This implies that once the reverse bumping sequence enters $-X'$, it will stay in $-X'$. Since we are restricting $\mread(U)$ to $X'$, the result follows.
\end{proof}

We then introduce the \emph{shifted plactic skew Schur $P$-function} of shape $\lambda / \mu$ to be the following element of $\QQ \shpl$:
\[
\mathcal{P}_{\lambda / \nu} = \sum_{\sh(T) = \lambda / \mu}[\rect(T)].
\]

\begin{conjecture}\label{conj}
$\mathcal{P}_{\lambda / \mu}$ belongs to the ring generated by the shifted plactic Schur $P$-functions.
\end{conjecture}

\begin{corollary}[of Conjecture \ref{conj}]
Fix a shifted Young tableau $U$ of shape $\nu$. The coefficient of $P_{\nu}$ in $P_{\lambda / \mu}$ is equal to the number of skew shifted Young tableaux $T$ with \hbox{$\pmix(\mread(T)) = U$}.
\end{corollary}

%PROVE THIS WHEN YOU HAVE TIME
\iffalse
For the moment we can prove a slightly weaker statement than Conjecture \ref{conj}.

\begin{theorem}
$\mathcal{P}_{\lambda / \mu}$ belongs to the ring generated by the plactic Schur functions.
\end{theorem}

\begin{proof}
In \cite[Proposition 7.1]{Ha}, Haiman gives a method (conjectured by Shor), in which the Robinson-Schensted insertion tableau of a word is obtained solely from its mixed insertion tableau. That is, $\pr(w)$ can be obtained solely from $\pmix(w)$, without knowing $w$. Thus, 

The projection $\pi(\mathcal{P}_{\lambda / \mu})$ (which lives in $\QQ \pl$) belongs to the ring generated by the plactic Schur functions $\mathcal{S}_{\mu}$. This will enable us to find a combinatorial interpretation for the coefficients in the Schur expansion of the skew Schur $P$-functions.
\end{proof}
\fi
%END COMMENT

\subsection{Semistandard decomposition tableaux}\label{ssdtsection}

The fact that the mixed reading word of $T$ can be decomposed into hook subwords, precisely of the same lengths as the rows of $T$, hints that arranging these words as rows of a shifted diagram would yield an interesting object. Based on the notion of a standard decomposition tableau introduced by W.~Kra\'skiewicz \cite{Kr} and further developed by T.~K.~Lam \cite{TKL}, we introduce the following.

\begin{definition}\label{ssdt}
A \emph{semistandard decomposition tableau} (SSDT) is a filling $R$ of a shape $\lambda = (\lambda_1, \ldots, \lambda_l)$ with elements of $X$, such that:
\begin{enumerate}
\item the word $u_i$ formed by reading the $i$-th row from left to right is a hook word of length $\lambda_i$, and
\item $u_i$ is a hook subword of maximum length in $u_l u_{l-1} \cdots u_i$, for $1 \le i \le l-1$.
\end{enumerate}

The \emph{reading word} of $R$ is $\word(R) = u_l u_{l-1} \cdots u_1$. The \emph{content} of an SSDT is the content of $\word(R)$.
\end{definition}

By Proposition \ref{ssdtproposition}, a filling $R$ of a shape $\lambda$ is an SSDT if and only if each of the fillings formed by two consecutive rows is an SSDT. By the definition, a filling $R$ is an SSDT of shape $\lambda$ if and only if $\word(R)$ is a shifted tableau word of shape $\lambda$.

\begin{example}\label{ssdtexample}
An SSDT, with its corresponding reading word:
\[
\ycenter
R = \young(96524,:511,::34)    \qquad    \word(R) = 34 \, \, 511 \, \, 96524.
\]
The content of $R$ is $(2,1,1,2,2,1,0,0,1)$.
\end{example}

\begin{remark}
One can see that an SSDT is a shifted analog of a (ordinary) Young tableau for the following reason. A word $w$ in the alphabet $X$ is the reading word of a Young tableau $T$ of shape $\lambda$ if and only if it is of the form $u_l u_{l-1} \cdots u_1$, where:
\begin{enumerate}
\item each word $u_i$ is weakly increasing,
\item the length of $u_i$ is $\lambda_i$, for $1 \le i \le l$, and
\item $u_i$ is a weakly increasing subword of maximal length in $u_l u_{l-1} \cdots u_i$, for $1 \le i \le l-1$.
\end{enumerate}
In this case the $u_i$ are precisely the rows of $T$.
\end{remark}

Shifted Young tableaux and SSDT share many properties, and many theorems about shifted Young tableaux can be proved more easily in the language of SSDT. The first important property they share is that they are in bijection.

\begin{theorem}\label{bijection}
Let $\mathcal{Y}(\lambda)$ be the set of shifted Young tableaux of shape $\lambda$. Let $\mathcal{D}(\lambda)$ be the set of SSDT of shape $\lambda$. The map
\[
\begin{array}{cccc}
\Phi:    & \mathcal{D}(\lambda)    & \rightarrow    & \mathcal{Y}(\lambda) \\
    & R                    & \mapsto        & \pmix (\word (R))
\end{array}
\]
is a bijection. Furthermore, $\word(R) = \mread(\Phi(R))$, i.e., $\Phi$ is a word preserving bijection.
\end{theorem}

For a more informal definition of $\Phi$, see Remark \ref{simplebijection}

As an example, the image under $\Phi$ of the SSDT in Example \ref{ssdtexample} is the shifted Young tableau of Example \ref{shiftedtableauexample}.

A special case of this bijection is for tableaux with only one row. The image of a shifted Young tableau of one row is the SSDT formed by reading the primed entries from right to left, and then the unprimed entries from left to right. For example,
\[
\ycenter \young(1{\twop}22{\threep}45{\sixp}) \stackrel{\Phi}{\mapsto} \young(63212245).
\]

In order to find the inverse of $\Phi$ directly, we define a semistandard version of Kra\'skiewicz insertion \cite{Kr}.

\begin{definition}[Semistandard Kra\'skiewicz (SK) insertion]\label{skinsertion}
Given a hook  word $u = y_1 \cdots y_k \cdots y_s$, where $u \downarrow = y_1 \cdots y_k$ and $u \uparrow = y_{k+1} \cdots y_s$, and a letter $x$, the insertion of $x$ into $u$ is the word $ux$ if $ux$ is a hook word, or the word $u'$ with an element $y$ that gets bumped out, as follows:
\begin{enumerate}
\item let $y_j$ be the leftmost element in $u \uparrow$ which is strictly greater than $x$;
\item replace $y_j$ by $x$;
\item let $y_i$ be the rightmost element in $u \downarrow$ which is less than or equal to $y_j$;
\item replace $y_j$ by $y_i$, to obtain $u'$, bumping $y = y_j$ out of $x$.
\end{enumerate}
To insert a letter $x$ into an SSDT $T$ with rows $u_1, u_2, \ldots, u_l$, one first inserts~$x = x_1$ into the top row~$u_1$. If an element~$x_2$ gets bumped, it will get inserted into the second row~$u_2$, and so on. The process terminates when an element~$x_i$ gets placed at the end of row~$u_i$.

The \emph{SK insertion tableau} of the word $w = w_1 \cdots w_n$, denoted $\psk(w)$, is obtained by starting with an empty shape and inserting the letters $w_1, \ldots, w_n$ from left to right, forming an SSDT.

The \emph{SK recording tableau} of $w$, denoted $\qsk(w)$, is the standard shifted Young tableau that records the order the elements have been inserted into $\psk(w)$. In other words, the shapes of $\psk(w_1 \cdots w_{i-1})$ and $\psk(w_1 \cdots w_i)$ differ by one box; $\qsk(w)$ has a letter $i$ on that box.
\end{definition}

\begin{example}\label{kinsertionexamplemaybe}
The following are the steps for inserting $3$ into the tableau
\[
\ycenter
\young(6542114,:63215,::522).
\]
In every step, $u \leftarrow a \seq b \leftarrow u' $ denotes the insertion of the letter $a$ into a row $u$, thus obtaining a new row $u'$, and a letter $b$ that gets bumped out and inserted into the next row. The horizontal bar separates $u \uparrow$ and $u \downarrow$.
\begin{eqnarray*}
65421 \, | \, 14 \leftarrow 3        & \seq & 4 \leftarrow 65\mathbf{4}21 \, | 1\mathbf{3}, \\
6321 | 5 \leftarrow 4            & \seq & 3 \leftarrow 6\mathbf{5}21 | \mathbf{4}, \\
52 | 2 \leftarrow 3            & \seq & 52 | 2 \mathbf{3},
\end{eqnarray*}
and the resulting tableau is
\[
\ycenter \young(6542113,:65214,::5223).
\]
\end{example}

\begin{example}\label{kinsertionexample1}
The word $w = 3415961254$ has the following SK insertion and recording tableau
\[
\ycenter
\psk(w) = \young(96524,:511,::34) \qquad \qsk(w) = \young(12459,:368,::7{\ten}).
\]
\end{example}

In Section \ref{proofssection} we define the \emph{standardization} of a word, of a shifted Young tableau, or of an SSDT, based on techniques introduced by Sagan \cite{Sa} and Haiman \cite{Ha}. Standardization enables us to translate results on permutations to results on words with repeated letters. We prove (Lemmas \ref{stanword}, \ref{stansdt} and \ref{stansyt}) that standardization commutes with shifted plactic equivalence, SK insertion, and mixed insertion. Thus, several of our theorems can be viewed as semistandard counterparts of theorems by Kra\'skiewicz \cite{Kr}, Lam \cite{TKLthesis} \cite{TKL}, and Sagan \cite{Sa}. Specifically, Theorem \ref{skisanssdt} is the semistandard extension of \cite[Theorem 5.2]{Kr}, while Theorem \ref{wordskequivalent}, Proposition \ref{samerecordingtableau}, and Theorem \ref{jdttheorem} correspond to \cite[Lemma 3.5, Corollary 3.6, Lemma 4.8]{TKL}, \cite[Theorem 3.34]{TKLthesis}, and \cite[Theorem 3.25]{TKLthesis}, respectively.

\begin{theorem}\label{skisanssdt}
SK insertion is a bijection between words in the alphabet $X$ and pairs of tableaux $(P,Q)$, where $P$ is an SSDT and $Q$ is a standard shifted Young tableau of the same shape as $P$.
\end{theorem}

\begin{theorem}\label{wordskequivalent}
Two words are shifted plactic equivalent if and only if they have the same SK insertion tableau. In particular, for a word $w$, $w \seq \word(\psk(w))$ (or equivalently, $\psk(\word(\psk(w))) = \psk(w)$). Furthermore, $\qsk(\word(\psk(w)))$ is the special recording tableau of the same shape as $\psk(w)$.
\end{theorem}

\begin{proposition}\label{samerecordingtableau}
The recording tableau of a word is the same under mixed insertion and SK insertion. Namely, $\qmix(w) = \qsk(w)$ for any word $w$ in the alphabet $X$.
\end{proposition}

\begin{remark}\label{simplebijection}
A more informal way to view $\Phi$ is as follows. Let $w$ be a word in the alphabet $X$. Then $\Phi$ sends $\pmix(w)$ to $\psk(w)$, and $\Psi$ does the opposite.
\end{remark}

\begin{proof}[Proof of Theorem \ref{bijection}]
We will prove that the inverse map is given by
\[
\begin{array}{cccc}
\Psi:    & \mathcal{Y}(\lambda)    & \rightarrow    & \mathcal{D}(\lambda) \\
    & T                    & \mapsto        & \psk (\mread (T)).
\end{array}
\]
Let $T \in \mathcal{Y}(\lambda)$. Then by Theorem \ref{wordskequivalent}, $\qsk(\mread(T))$ is the special recording tableau of shape~$\lambda$. Therefore, $\Psi(T) = \psk(\mread(T))$ has shape $\lambda$, i.e., $\Psi(T) \in \mathcal{D}(\lambda)$. Similarly, if $R \in \mathcal{D}(\lambda)$, then $\Phi(R) \in \mathcal{Y}(\lambda)$.

Now, note that for $T \in \mathcal{Y}(\lambda)$,
\begin{eqnarray*}
\Phi ( \Psi (T) )    & = & \pmix ( \word (\psk (\mread(T) ) ) ) \\
            & = & \pmix (\mread(T)) \\
            & = & T.
\end{eqnarray*}
And similarly, for $R \in \mathcal{D}(\lambda)$, $\Psi( \Phi(R) ) = R$. Therefore, $\Phi$ is a bijection.
\end{proof}

\begin{proof}[Proof of Theorem \ref{tableauwordtheorem}]
By Theorem \ref{bijection}, $w$ is a shifted tableau word if and only if $w = \word(R)$ for some SSDT $R$. Therefore, (1) follows from the definition of an SSDT.

(2) follows from the observation that if $u_l, \ldots, u_1$ are rows of an SSDT, then so are $u_l, \ldots, u_i$ for all $1 \le i \le l$.
\end{proof}

Sagan \cite{Sa} and Worley \cite{Wo} have introduced the \emph{Sagan-Worley insertion}, a correspondence between words and pairs of shifted tableaux. For a word $w$, we denote the \emph{Sagan-Worley insertion and recording tableaux} by $\pw(w)$ and $\qw(w)$, respectively.

Both mixed insertion and Sagan-Worley insertion can be generalized to \emph{biwords}, namely two-rowed arrays in which the columns are arranged lexicographically, with priority given to the element in the top row. In this setting, one inserts the elements from the bottom row left to right, while recording the corresponding elements from the top row. The ordinary mixed insertion is obtained if the top row has the numbers $1, 2, 3, \ldots$. Thus, one can identify a word $u$ with the two-rowed array where $u$ is in the bottom row, and the numbers $1,2,3, \ldots$ form in the top row. The inverse of $u$, denoted $u^{-1}$ is the biword obtained by switching the two rows in $u$, and organizing the columns lexicographically with priority given to the element on the top row. Haiman has proved that mixed insertion is dual to Sagan-Worley insertion in the following sense: $\pmix(u) = \qw(u^{-1})$ and $\qmix(u) = \qw(u^{-1})$. This was extended to biwords by Fomin \cite[p. 291]{Fo}.

Thus, an alternative way to define shifted plactic equivalence (hence the shifted plactic monoid) is as follows: Two words $u$ and $v$ are shifted plactic equivalent if and only if $\qw(u^{-1}) = \qw(v^{-1})$.

In light of this, we obtain the following theorem as a special case of \cite[Theorem 10.2]{Sa}, and also as a semistandard extension of \cite[Lemma 4.8]{Kr}. 

\begin{theorem}\label{shiftedgreenestheorem}
The length of the longest hook subword of a word $w$ is equal to the length of the top row of $\pmix(w)$ (or equivalently of $\psk(w)$).
\end{theorem}

Note that an ordinary shape $\lambda = (\lambda_1, \lambda_2, \ldots, \lambda_l)$ can be considered a skew shifted shape $(\lambda + \delta) / \delta$, where $\delta = (l, l-1, \ldots, 1)$. The following result, which can be viewed as the semistandard version of \cite[Theorem 3.25]{TKLthesis}, relates the mixed (or SK) recording tableau of a word to its Robinson-Schensted-Knuth recording tableau.

\begin{theorem}\label{jdttheorem}
Let $w$ be a word in the alphabet $X$. Then the tableau $\qmix(w)$ (which is the same as $\qsk(w)$) can be obtained by treating $\qr(w)$ as a skew shifted Young tableau and applying shifted jeu de taquin slides to it to get a standard shifted Young tableau.
\end{theorem}

\section{Proofs}\label{proofssection}

\begin{definition}
A word of length $n$ is called a \emph{permutation} if it contains each of the letters $1,2,\ldots, n$ each exactly once. A shifted Young tableau of content $(1,1,\ldots,1)$ will be called a \emph{Haiman tableau}. In other words, a Haiman tableau is a standard shifted tableau, possibly with some primed off-diagonal entries.

Let $w = w_1 w_2 \ldots w_m$ be a word. Let $(\alpha_1, \ldots, \alpha_m)$ be the \emph{content} of $w$, i.e., $\alpha_i$ letters in $w$ are equal to $i$. The \emph{standardization} of $w$, denoted $\stan(w)$ is the permutation obtained by relabelling the elements labelled~$i$ by $\alpha_1+\cdots + \alpha_{i-1}+1, \ldots, \alpha_1 + \cdots + \alpha_i$, from left to right.

The \emph{standardization} of an SSDT~$R$, denoted $\stan(R)$ is the filling obtained by applying the same procedure used in the standardization of a word, namely, taking the order of the elements from $\word(R)$.  Note that in general, any filling of a shifted shape with letters in $X$ can be standardized, using this method.

The standardization of a shifted Young tableau $T$, denoted $\stan(T)$ is the Haiman tableau of the same shape, obtained as follows: Let $\alpha = (\alpha_1, \ldots, \alpha_m)$ be the content of $T$. For every $i$, the entries labelled $i'$ or $i$ are relabelled $\alpha_1 + \ldots + \alpha_{i-1}+1, \alpha_1 + \ldots + \alpha_{i-1}+2, \ldots, \alpha_1 + \ldots + \alpha_{i-1} + \alpha_i$. One assigns these values in increasing order, starting from the boxes labelled $i'$, from top to bottom, and then moving to the boxes labelled $i$, from left to right. If the old element in a box was primed, so is the new one. This procedure has been suggested by Haiman \cite{Ha}.
\end{definition}

\begin{example}\label{standardizationexample}
Standardizations of a word $w$, a shifted Young tableau $T$, and an SSDT $R$:
\[
w = 23314211, \qquad \stan(w) = 46718523,
\]
\[ \ycenter
R= \young(4211,:313,::2), \qquad
\stan(R)= \young(8523,:617,::4),
\]
\[ \ycenter
T= \young(111{\twop},:2{\threep}4,::3), \qquad
\stan(T)= \young(123{\fourp},:5{\sixp}8,::7).
\]
\end{example}

In his study of decompositions of reduced words in the hyperoctahedral group $B_N$, Kra\'skiewicz \cite{Kr} has introduced \emph{standard decomposition tableaux} (SDT), which are defined in the same way as SSDT, with the extra condition that the reading word must be a reduced word in $B_N$. He has also introduced the \emph{Kra\'skiewicz correspondence}, which assigns to every reduced word in $B_N$, a pair of tableaux consisting of the Kra\'skiewicz insertion tableau (an SDT), and the Kra\'skiewicz recording tableau (a standard Young tableau).

We do not use the full power of Kra\'skiewicz insertion, but we use it for permutations, since a permutation is always a reduced word in $B_N$, for some $N$. (Here the letter $i$ stands for the $i$-th generator of $B_N$ as a Coxeter group.) We abuse notation and call $\psk(w)$ and $\qsk(w)$ the \emph{Kra\'skiewicz insertion} and \emph{recording tableaux}, respectively, and use the term SDT to refer to a semistandard decomposition tableau whose reading word is a permutation. We also note that for permutations, the shifted plactic relations are equivalent to the \emph{B-Coxeter-Knuth relations} in \cite{TKL}, and to the dual equivalence relations given in Corollary 3.2 in \cite{HaD}.

\begin{lemma}\label{stanword}
Two words $u$ and $v$ in the alphabet $X$ are shifted plactic equivalent if and only if they have the same content, and $\stan(u) \seq \stan(v)$.
\end{lemma}
\begin{proof}
It is enough to check this for relations (\ref{sk1})--(\ref{sk8}). For example, for (\ref{sk1}), there are four possible type of words that will have $abdc$ as their standardization, with $a < b < c < d$, namely $aaba, aacb, abcb, abdc$. Similarly for $adbc$, we have $abaa, abac, acab, adbc$. These eight words are paired by relation (\ref{sk1}) as follows:
\begin{eqnarray*}
aaba        & \seq    & abaa; \\
aacb        & \seq    & acab; \\
abcb        & \seq     & acbb; \\
abdc        & \seq    & adbc.
\end{eqnarray*}
Checking the other seven relations is routine, and left for the reader.
\end{proof}

\begin{lemma}
Standardization respects the property of being a shifted tableau word, an SSDT, or a semistandard shifted Young tableau. More precisely:
\begin{itemize}
\item a word $w$ is a shifted tableau word if and only if $\stan(w)$ is;
\item a filling of a shifted shape $R$ with letters in $X$ is an SSDT if and only if $\stan(R)$ is an SDT;
\item a filling of a shifted shape $T$ with letters in $X'$ is a semistandard shifted Young tableau if and only if $\stan(T)$ is a Haiman tableau.
\end{itemize}
\end{lemma}
\begin{proof}
The first two points follow from the observation that a word $w$ is a hook word if and only if its standardization $\stan(w)$ is a hook word. The third one is a routine check of the properties that define a shifted Young tableau.
\end{proof}

\begin{lemma}\label{stansdt}
Let $w$ be a word and $R$ an SSDT. Then,
\begin{eqnarray*}
\psk(\stan(w))    & = &  \stan(\psk(w)); \\
\qsk(\stan(w))    & = & \qsk(w); \\
\word(\stan(R))    & = & \stan(\word(R)).
\end{eqnarray*}
\end{lemma}
\begin{proof}
Cases (1) and (2) of the SK insertion algorithm treat the two following scenarios in the same fashion:
\begin{itemize}
\item when a letter $i$ gets bumped by a letter $j > i$;
\item when a letter $i$ gets bumped by another letter $i$.
\end{itemize}
Therefore, if two letters $i$ are in an SSDT $R$, the algorithm treats them as if the rightmost one in the reading word was larger. Furthermore, the $i$ at the right will always remain at the right of the $i$ at the left, i.e., they will always maintain their relative position with respect to each other. The third equation is a rewording of the definition of standardization of an SSDT.
\end{proof}

\begin{lemma}\label{stansyt}
Let $w$ be a word and $T$ a shifted Young tableau. Then,
\begin{eqnarray*}
\pmix(\stan(w))        & = &  \stan(\pmix(w)); \\
\qmix(\stan(w))        & = & \qmix(w); \\
\mread(\stan(T))    & = & \stan(\mread(T)).
\end{eqnarray*}
\end{lemma}
\begin{proof}
Let $w = w_1 \cdots w_l$, and $T = \pmix(w)$. The result follows from the next claim. Say $w$ has two letters equal to $i$, say, $w_a = w_b = i$ for $a < b$. We say that a letter is primed or unprimed depending of the state of its corresponding entry in $T$. Let $r_a$ and $c_a$ be the row and column where $w_a$ gets located in $T$, and $r_b$, $c_b$ the row and column where $w_b$ gets located. Then either
\begin{itemize}
\item $w_a$ and $w_b$ are unprimed and $c_a < c_b$,
\item $w_a$ and $w_b$ are primed and $r_a < r_b$, or
\item $w_a$ is primed and $w_b$ is unprimed.
\end{itemize}
The claim is proved as follows. Assume that both $w_a$ and $w_b$ are unprimed. As they both have the label $i$, then whenever $w_b$ reaches the row where $w_a$ is, it must be at its right. No number smaller than $i$ can bump $w_b$ before bumping $w_a$. Therefore, $w_a$ always remains strictly at the left of $w_b$, i.e., $c_a < c_b$.

If both $w_a$ and $w_b$ are primed, namely, they have the label $i'$, it means they both reached the diagonal and later got bumped out of the diagonal. An analogous argument to the above one, replacing columns for rows, will show that $w_a$ must be in a row above $w_b$, i.e., $r_a < r_b$.

It remains to show that the case where $w_a$ is unprimed and $w_b$ is primed can not occur. For this to happen, $w_b$ must reach the diagonal, but $w_a$ must not. But this is impossible since $w_b$ must remain at the right of $w_a$ when they are both unprimed.
\end{proof}

\begin{proof}[Proof of Theorem \ref{monoidtheorem}]
Let $u, v$ be words in the alphabet $X$. By Lemma \ref{stanword}, $u \seq v$ if and only if they have the same content and $\stan(u) \seq \stan(v)$.

By Corollary 3.2 in \cite{HaD}, $\stan(u) \seq \stan(v)$ if and only if $\pmix(\stan(u)) = \pmix(\stan(v))$. By Lemma \ref{stansyt}, the latter is equivalent to $\stan(\pmix(u)) = \stan(\pmix(v))$.

By the definition of standardization of shifted Young tableaux, it is clear that $\pmix(u) = \pmix(v)$ if and only if $u$ and $v$ have the same content and $\stan(\pmix(u)) = \stan(\pmix(v))$. The result then follows.
\end{proof}

The following theorem is the semistandard extension of \cite[Lemma 3.11]{TKL}.

\begin{lemma}\label{uni}
Let $R$ be an SSDT and $w = w_1 \cdots w_l$ a word in the alphabet $X$. Let $S$ be the SSDT obtained by SK inserting $w_1, \ldots, w_l$ into $R$. Let $\mu = \sh(R)$ and $\lambda = \sh(S)$. Then $w$ is a hook word if and only if the entries in the standard skew tableau of shape $\lambda / \mu$ that records the insertion of $w_1, \ldots, w_l$ in $R$ is a vee.
\end{lemma}

\begin{proof}[Proof of Theorem \ref{commutetheorem}]
It suffices to prove that the shifted plactic Schur $P$-functions satisfy the Pieri rule, namely
\[ \mathcal{P}_{\mu} \mathcal{P}_{(k)} = \sum_{\lambda} 2^{c(\lambda / \mu) - 1} \mathcal{P}_{\lambda},
\]
where $\lambda$ runs over all strict partitions such that $\lambda / \mu$ is a (possibly disconnected) border strip with $c(\lambda / \mu)$ connected components.

It is well known that a skew tableau $U_{\lambda / \mu}$ of shape ${\lambda / \mu}$ is a vee if and only if it rectifies to the single row with entries $i, i+1, \ldots, k$ under jeu de taquin, and that the number of vees of this shape with a fixed content is exactly $2^{c(\lambda / \mu)-1}$, where ${c(\lambda / \mu)}$ is the number of connected components of ${\lambda / \mu}$.

Recall that $\mathcal{D}(\lambda)$ is the set of SSDT of shape $\lambda$, where we denote $\mathcal{D}((k))$ by $\mathcal{D}(k)$.  Let $\mu \oplus k$ be the set of shifted shapes $\lambda$ such that $\lambda / \mu$ is a (possibly disconnected) border strip of size $k$. Let $\mathcal{V}(\lambda / \mu)$ be the set of vees of shape $\lambda / \mu$ filled with the entries $|\mu|+1, |\mu|+2, \cdots. |\lambda|$.

We will prove that $\mathcal{D}(\mu) \times \mathcal{D}(k)$ and $\bigcup_{\lambda \in \mu \oplus k} \left( \mathcal{D}(\lambda) \times \mathcal{V}(\lambda / \mu) \right)$ are in bijection, and moreover, that the content of the elements in $\mathcal{D}(\mu)$ and $\mathcal{D}(k)$ adds up to the content of the element in $\mathcal{D}(\lambda)$ in their image. Thus, the theorem will follow since $\mathcal{P}_{\lambda}$ is the formal sum of all elements in $\mathcal{D}(\lambda)$.

Consider the map
\[
\Phi : \mathcal{D}(\mu) \times \mathcal{D}(k) \rightarrow \bigcup_{\lambda \in \mu \oplus k} \left( \mathcal{D}(\lambda) \times \mathcal{V}(\lambda / \mu) \right)
\]
defined as follows. Given $S_{\mu} \in \mathcal{D}(\mu)$ and $T_{k} \in \mathcal{D}(k)$, insert the elements of $\word(T_{k})$ into $S_{\mu}$ from left to right to obtain an SSDT $R_{\lambda}$ of some shape $\lambda$. Since $\word(T_{k})$ is a hook word then by Lemma \ref{uni}, $\lambda \in \mu \oplus k$, and the standard skew tableau $V_{\lambda / \mu}$ that records this insertion is in $\mathcal{V}(\lambda / \mu)$. We define $\Phi(S_{\mu},T_{k}) = (R_{\lambda},V_{\lambda / \mu})$. Clearly, the content of $S_{\mu}$ and $T_{k}$ add up to the content of~$R_{\lambda}$.

The inverse map
\[
\Psi : \bigcup_{\lambda \in \mu \oplus k} \left( \mathcal{D}(\lambda) \times \mathcal{V}(\lambda / \mu) \right) \rightarrow \mathcal{D}(\mu) \times \mathcal{D}(k)
\]
is defined in a very similar way, by removing the elements of $R_{\lambda}$ in the order given by the vee $V_{\lambda / \mu}$, to obtain $S_{\nu} \in \mathcal{D}(\mu)
$ and $T_{k} \in \mathcal{D}(k)$. Since this map removes the elements that were inserted in the definition of $\Phi$, it is clear that the composition of these two maps is the identity.

The following example illustrates the bijection:
\[
\ycenter
\Phi \left(
\young(4211,:313,::2) , \young(53122)
\right)
=
\left( \young(5311122,:3124,::23) ,
\young(~~~~145,:~~~2,::~3)
%\young(~~~~9{\twelve}{\thirteen},:~~~{\ten},::~{\eleven})
\right). \qedhere
\]
\end{proof}

The following definition appears in \cite{TKL} and \cite{Sa}.

\begin{definition}
Let $U$ be a standard shifted Young tableau. Define $\Delta(U)$ to be the tableau obtained by applying the following operations:
\begin{enumerate}
\item remove the entry $1$ from $U$;
\item apply a jeu de taquin slide into this box;
\item deduct $1$ from the remaining boxes.
\end{enumerate}
\end{definition}

The following lemma follows straight from \cite[Theorem 4.14]{TKL} and the standardization Lemmas \ref{stanword} and \ref{stansdt}.

\begin{lemma}\label{oneevacuation}
Let $w_1 w_2 \cdots w_l$ be a word. Then, 
\[
\qsk(w_2 \cdots w_l) = \Delta \qsk(w_1 w_2 \cdots w_l).
\]
\end{lemma}

\begin{proof}[Proof of Lemma \ref{slrrulesarethesame}]
The bijection $\Phi$ in the proof of Theorem \ref{commutetheorem} is a special case of the following bijection:
\[
\Phi: \mathcal{D}(\mu) \times \mathcal{D}(\nu) \rightarrow \bigcup_{\lambda} \left( \mathcal{D}(\lambda) \times \mathcal{V}(\lambda / \mu) \right),
\]
where $\mathcal{V}(\lambda / \mu)$ is the set of standard shifted skew tableau which rectify to the special recording tableau of shape $\nu$. This bijection is described in the exact same way, so we will not go into much detail. However, it is necessary to check that $\mathcal{V}(\lambda / \mu)$ is indeed the set we claim.

For a skew tableau $T$ and an integer $k$, let $T+k$ be the skew tableau of the same shape, where all the entries are raised by $k$. Let $S_{\mu} \in \mathcal{D}(\mu)$, $T_{\nu} \in \mathcal{D}(\nu)$. Let $\lambda$ be the shape of $\pmix(\mread(S_{\mu}) \mread(T_{\nu}))$. Let $U_{\lambda / \mu} \in \mathcal{V}(\lambda / \mu)$ be the standard shifted skew tableau of shape $\lambda / \mu$ which records the order in which the entries of $\mread(T_{\nu})$ get inserted into $S_{\mu}$. Note that $U_{\lambda / \mu} + |\mu|$ is precisely the subtableau of $\qmix(\mread(S_{\mu}) \mread(T_{\nu}))$ corresponding to the shape $\lambda / \mu$. By Lemma \ref{oneevacuation} applied repeatedly, one can see that $\qmix(\mread(T_{\nu}) = \Delta^{|\mu|} (U_{\lambda / \mu} + |\mu|)$. But $\Delta^{|\mu|} (U_{\lambda / \mu} + |\mu|)$ is nothing more than the restriction of $U_{\lambda / \mu}$. Since $\qmix(\mread(T_{\nu}))$ is the standard recording tableau of shape $\nu$, by definition, then $U_{\lambda / \mu} \in \mathcal{V}(\lambda / \mu)$.

Lemma \ref{slrrulesarethesame} then follows, if one lets $u = \mread(S_{\mu})$, $v = \mread(T_{\nu})$, and $Q = \qmix(\mread(T_{\nu}))$, i.e., the special recording tableau of shape $\nu$.
\end{proof}

\begin{proof}[Proof of Theorems \ref{pschurintermsofschur} and \ref{usstem}]
Let $U_{\lambda}$ be the special recording tableau of shifted shape~$\lambda$. Let $\mathcal{H}(\lambda,\mu)$ be the set of (ordinary) standard Young tableaux of shape $\mu$ which rectify to $U_{\lambda}$ (the proof is the same if $U_{\lambda}$ is any other standard shifted Young tableau of shape~$\lambda$).

Let $\langle P_{\mu} \rangle$ be a plactic class of shape $\mu$, and $\mathcal{G}(\lambda,P_{\mu})$ be the set of shifted plactic classes $[T_{\lambda}]$ of shifted shape $\lambda$ such that $\pi([T_{\lambda}]) = \langle P_{\mu} \rangle$. We will prove that the size of $\mathcal{G}(\lambda,P_{\mu})$ does not depend on $P_{\mu}$ (only on $\lambda$ and $\mu$), by finding a bijection $\Theta$ between $\mathcal{G}(\lambda,P_{\mu})$ and $\mathcal{H}(\lambda,\mu)$.

Define the maps
\[
\begin{array}{cccc}
\Theta:    & \mathcal{G}(\lambda,P_{\mu})    & \rightarrow    & \mathcal{H}(\lambda,\mu) \\
    & [T_{\lambda}]                    & \mapsto        & \qr (\mread (T_{\lambda}))
\end{array}
\]
and
\[
\begin{array}{cccc}
\Gamma:    & \mathcal{H}(\lambda,\mu)    & \rightarrow    & \mathcal{G}(\lambda,P_{\mu}) \\
        & Q                        & \mapsto        & [\pmix(w)]
\end{array}
\]
where $w$ is the word in the alphabet $X$ such that $\pr(w) = P_{\mu}$ and $\qr(w) = Q$.

First assume that $[T_{\lambda}] \in \mathcal{G}(\lambda,P_{\mu})$. Let $w = \mread(T_{\lambda})$ (i.e., the canonical representative of $[T_{\lambda}]$), so by definition, $\qmix(w) = U_{\lambda}$. By Theorem \ref{jdttheorem}, $\qr(w)$ rectifies to $U_{\lambda}$, so $\Theta([T_{\lambda}]) = \qr(\mread(T_{\lambda})) = \qr(w) \in \mathcal{H}(\lambda,\mu)$.

Now assume that $Q \in \mathcal{H}(\lambda,\mu)$. Let $w$ be such that $\pr(w) = P_{\mu}$ and $\qr(w) = Q$. Then, again by Theorem \ref{jdttheorem}, $Q = \qr(w)$ rectifies to $\qmix(w)$, but as $Q \in \mathcal{H}(\lambda,\mu)$, then $\qmix(w)$ has shape $\lambda$. Furthermore, as $\pr(w) = P_{\mu}$, then $\Gamma(Q) = [\pmix(w)] \in \mathcal{G}(\lambda,P_{\mu})$.

To prove that $\Theta$ and $\Gamma$ are inverse maps, again let $Q \in \mathcal{H}(\lambda / \mu)$. Thus, $\Gamma(Q) = [\pmix(w)]$ where $w$ is the word such that $\pr(w) = P_{\mu}$ and $\qr(w) = Q$. Since $Q$ rectifies to $U_{\lambda}$, $\qr(w)$ rectifies to $\qmix(w)$ (by Theorem \ref{jdttheorem}), and $\qmix(w) = Q$, then $\qmix(w) = U_{\lambda}$. Thus, $w$ is the canonical representative of $[\pmix(w)]$, which implies that $\mread(\pmix(w)) = w$. Therefore,
\begin{eqnarray*}
\Theta(\Gamma(Q))	& = & \Theta([\pmix(w)]) \\
				& = & \qr(\mread(\pmix(w))) \\
				& = & \qr(w) \\
				& = & Q.
\end{eqnarray*}
The proof that $\Gamma(\Theta([T_{\lambda}])) = [T_{\lambda}]$ is similar.

%Notice that if $Q = \Theta([T_{\lambda}])$ for some $[T_{\lambda}] \in \mathcal{H}(\lambda,\mu)$, then $\Gamma(Q)$ is the shifted plactic class of the word $w$ such that $\pr(w) = P_{\mu}$ and $\qr(w) = \qr(\mread(T_{\lambda}))$. But the word $\mread(T_{\lambda})$ also satisfies these two conditions, since $[T_{\lambda}] \in \mathcal{G}(\lambda,P_{\mu})$ and a word is uniquely defined by its Robinson-Schensted-Knuth insertion and recording tableau. Therefore $\Theta$ and $\Gamma$ are inverse maps.

Since $\mathcal{H}(\lambda,\mu)$ clearly does not depend on the choice of $P_{\mu}$, but only on the shape $\mu$, then neither does the number of shifted plactic classes $[T_{\lambda}]$ of shifted shape $\lambda$ such that $\pi([T_{\lambda}]) = \langle P_{\mu} \rangle$. This number is precisely $g^{\lambda}_{\mu}$. Furthermore, it is also equal to the size of $\mathcal{H}(\lambda,\mu)$, which proves Theorem \ref{usstem}.
\end{proof}

\begin{proof}[Proof of Proposition \ref{ssdtproposition}]
Let $u_1, \ldots, u_l$ be hook words. We use the fact that $w = u_1 \cdots u_l$ is a shifted tableau word if and only if the tableau formed by the rows $u_1, \ldots, u_l$ from top to bottom is an SSDT.

Clearly, $(2)$ implies $(2')$. We will prove the converse.

It suffices to prove the claim for $l = 3$. Namely, that for three hook words $u_1, u_2, u_3$ of lengths $a > b > c$, respectively, if $u_3 u_2$ and $u_2 u_1$ are shifted tableau words, then so is $u_3 u_2 u_1$. Equivalently, we will prove that if the filling of the shape $(a,b)$ with rows $u_1$ and $u_2$, and the filling of the shape $(b,c)$ with rows $u_2$ and $u_3$ are both SSDT, then the filling of the shape $(a,b,c)$ with rows $u_1$, $u_2$, and $u_3$ is an SSDT as well.

By Theorem \ref{wordskequivalent}, $\qsk(u_3 u_2)$ is the special recording tableau of shape $(b,c)$. Let us assume that the tableu formed by the rows $u_1$, $u_2$, $u_3$ is not an SSDT. This means that the longest hook subword in the word $u_3 u_2 u_1$ is of length $d$, for some $d > a$. Therefore, by Theorem \ref{shiftedgreenestheorem}, the top row of $\qsk(u_3 u_2 u_1)$ is of length $d$. By Lemma \ref{oneevacuation} applied repeatedly, $\qsk(u_2 u_1) = \Delta^c \qsk(u_3 u_2 u_1)$. But note that since $\qsk(u_3 u_2)$ is a subtableau of $\qsk(u_3 u_2 u_1)$, and $\Delta^c \qsk(u_3 u_2) = \qsk(u_2)$ which is a row of length $c$, then the top row of $\Delta^c (\qsk(u_2 u_1))$ has length $d$ (because applying $\Delta^c$ to $\qsk(u_3 u_2 u_1)$ will not alter the top length of the top row, since it doesn't alter the top row of $\qsk(u_3 u_2)$). This contradicts the assumption that $u_2 u_1$ is a shifted tableau word, completing the proof.
\end{proof}

\vspace{-.2in}

\appendix

\section*{Appendix. Shifted plactic classes of 4-letter words}

The following tables show all types of 4-letter words $w$, together with their corresponding mixed insertion tableaux $\pmix(w)$, and their Robinson-Schensted-Knuth insertion tableau $\pr(w)$. By convention, ~$a < b < c < d$.

\setlength{\tabcolsep}{12pt}
\noindent
\[
\begin{tabular}{|c|c|c|}\hline
& & \\[-.1in]
$w$            & $\pmix(w)$            & $\pr(w)$         \\[.04in] \hline
& & \\[-.1in]
$aaaa$        & \young(aaaa)            & \young(aaaa) \\[.04in] \hline
\end{tabular}
\]
\vspace{-.2in}
\[
\begin{tabular}{|c|c|c|}\hline
& & \\[-.1in]
$w$            & $\pmix(w)$            & $\pr(w)$         \\[.04in] \hline
& & \\[-.1in]
$aaab$        & \young(aaab)                & \young(aaab) \\[.04in] \hline
& & \\[-.1in]
$aaba$        & \multirow{2}{*}{$\young(aaa,:b)$}    & \multirow{3}{*}{$\young(aaa,b)$} \\[.04in] \cline{1-1}
& & \\[-.1in]
$abaa$        &                             &\\[.04in] \cline{1-2}
& & \\[-.1in]
$baaa$        & $\young(aaa{\bprime})$            & \\[.04in] \hline
\end{tabular}
\hspace{.51in}
\begin{tabular}{|c|c|c|}\hline
& & \\[-.1in]
$w$            & $\pmix(w)$            & $\pr(w)$         \\[.04in] \hline
& & \\[-.1in]
$abbb$        & \young(abbb)            & \young(abbb) \\[.04in]\hline
& & \\[-.1in]
$bbba$        & \multirow{2}{*}{$\young(a{\bprime}b,:b)$}    & \multirow{3}{*}{$\young(abb,b)$} \\[.04in]\cline{1-1}
& & \\[-.1in]
$bbab$        &                             &\\[.04in]\cline{1-2}
& & \\[-.1in]
$babb$        & $\young(a{\bprime}bb)$            & \\[.04in]\hline
\end{tabular}
\]
\vspace{-.2in}
\[
\begin{tabular}{|c|c|c|}\hline
& & \\[-.1in]
$w$            & $\pmix(w)$            & $\pr(w)$         \\[.04in] \hline
& & \\[-.1in]
$aabb$        & \young(aabb)                & \young(aabb) \\[.04in] \hline
& & \\[-.1in]
$abba$        & \multirow{2}{*}{$\young(aab,:b)$}    & \multirow{3}{*}{$\young(aab,b)$} \\[.04in] \cline{1-1}
& & \\[-.1in]
$abab$        &                             &\\[.04in] \cline{1-2}
& & \\[-.1in]
$baab$        & $\young(aa{\bprime}b)$            & \\[.04in] \hline

& & \\[-.1in]
$baba$        & \multirow{2}{*}{$\young(aa{\bprime},:b)$}    & \multirow{2}{*}{$\young(aa,bb)$} \\[.04in]\cline{1-1}
& & \\[-.1in]
$bbaa$        &                            & \\[.04in]\hline
\end{tabular}
\]
\begin{tabular}{|c|c|c|}\hline
& & \\[-.1in]
$w$            & $\pmix(w)$            & $\pr(w)$         \\[.04in] \hline
& & \\[-.1in]
$aabc$        & \young(aabc)            & \young(aabc) \\[.04in]\hline
& & \\[-.1in]
$abca$        & \multirow{2}{*}{$\young(aac,:b)$}    & \multirow{3}{*}{$\young(aac,b)$} \\[.04in]\cline{1-1}
& & \\[-.1in]
$abac$        &                             &\\[.04in]\cline{1-2}
& & \\[-.1in]
$baac$        & $\young(aa{\bprime}c)$            & \\[.04in]\hline

& & \\[-.1in]
$aacb$        & \multirow{2}{*}{$\young(aab,:c)$}    & \multirow{3}{*}{$\young(aab,c)$} \\[.04in]\cline{1-1}
& & \\[-.1in]
$acab$        &                             &\\[.04in]\cline{1-2}
& & \\[-.1in]
$caab$        & $\young(aab{\cprime})$            & \\[.04in]\hline

& & \\[-.1in]
$acba$        & \multirow{2}{*}{$\young(aa{\cprime},:b)$}    & \multirow{3}{*}{$\young(aa,b,c)$} \\[.04in]\cline{1-1}
& & \\[-.1in]
$caba$        &                             &\\[.04in]\cline{1-2}
& & \\[-.1in]
$cbaa$        & $\young(aa{\bprime}{\cprime})$            & \\[.04in]\hline

& & \\[-.1in]
$baca$        & \multirow{2}{*}{$\young(aa{\bprime},:c)$}    & \multirow{2}{*}{$\young(aa,bc)$} \\[.04in]\cline{1-1}
& & \\[-.1in]
$bcaa$        &                            & \\[.04in]\hline
\end{tabular}
\begin{tabular}{|c|c|c|}\hline
& & \\[-.1in]
$w$            & $\pmix(w)$            & $\pr(w)$         \\[.04in] \hline
& & \\[-.1in]
$abbc$        & \young(abbc)            & \young(abbc) \\[.04in]\hline
& & \\[-.1in]
$bbca$        & \multirow{2}{*}{$\young(a{\bprime}c,:b)$}    & \multirow{3}{*}{$\young(abc,b)$} \\[.04in]\cline{1-1}
& & \\[-.1in]
$bbac$        &                             &\\[.04in]\cline{1-2}
& & \\[-.1in]
$babc$        & $\young(a{\bprime}bc)$            & \\[.04in]\hline

& & \\[-.1in]
$abcb$        & \multirow{2}{*}{$\young(abb,:c)$}    & \multirow{3}{*}{$\young(abb,c)$} \\[.04in]\cline{1-1}
& & \\[-.1in]
$acbb$        &                             &\\[.04in]\cline{1-2}
& & \\[-.1in]
$cabb$        & $\young(abb{\cprime})$            & \\[.04in]\hline

& & \\[-.1in]
$bcba$        & \multirow{2}{*}{$\young(a{\bprime}{\cprime},:b)$}    & \multirow{3}{*}{$\young(ab,b,c)$} \\[.04in]\cline{1-1}
& & \\[-.1in]
$cbba$        &                             &\\[.04in]\cline{1-2}
& & \\[-.1in]
$cbab$        & $\young(a{\bprime}b{\cprime})$            & \\[.04in]\hline

& & \\[-.1in]
$bacb$        & \multirow{2}{*}{$\young(a{\bprime}b,:c)$}    & \multirow{2}{*}{$\young(ab,bc)$} \\[.04in]\cline{1-1}
& & \\[-.1in]
$bcab$        &                            & \\[.04in]\hline
\end{tabular}
\linebreak
\hspace{.1in}
\begin{tabular}{|c|c|c|}\hline
& & \\[-.1in] %LINE BELOW HAS ARTIFICIAL HSPACING TO FIX MISMATCH
\hspace{.04in} $w$ \hspace{.04in}           & $\pmix(w)$            & $\pr(w)$         \\[.04in] \hline
& & \\[-.1in]
$abcc$        & \young(abcc)            & \young(abcc) \\[.04in]\hline
& & \\[-.1in]
$bcca$        & \multirow{2}{*}{$\young(a{\bprime}c,:c)$}    & \multirow{3}{*}{$\young(acc,b)$} \\[.04in]\cline{1-1}
& & \\[-.1in]
$bcac$        &                             &\\[.04in]\cline{1-2}
& & \\[-.1in]
$bacc$        & $\young(a{\bprime}cc)$            & \\[.04in]\hline

& & \\[-.1in]
$accb$        & \multirow{2}{*}{$\young(abc,:c)$}    & \multirow{3}{*}{$\young(abc,c)$} \\[.04in]\cline{1-1}
& & \\[-.1in]
$acbc$        &                             &\\[.04in]\cline{1-2}
& & \\[-.1in]
$cabc$        & $\young(ab{\cprime}c)$            & \\[.04in]\hline

& & \\[-.1in]
$ccba$        & \multirow{2}{*}{$\young(a{\bprime}{\cprime},:c)$}    & \multirow{3}{*}{$\young(ac,b,c)$} \\[.04in]\cline{1-1}
& & \\[-.1in]
$cbca$        &                             &\\[.04in]\cline{1-2}
& & \\[-.1in]
$cbac$        & $\young(a{\bprime}b{\cprime})$            & \\[.04in]\hline

& & \\[-.1in]
$cacb$        & \multirow{2}{*}{$\young(ab{\cprime},:c)$}    & \multirow{2}{*}{$\young(ab,cc)$} \\[.04in]\cline{1-1}
& & \\[-.1in]
$ccab$        &                            & \\[.04in]\hline
\end{tabular}
%\]
\vspace{1in}
\[
\begin{tabular}{|c|c|c|}\hline
& & \\[-.1in]
$w$            & $\pmix(w)$            & $\pr(w)$         \\[.04in] \hline
& & \\[-.1in]
$abcd$        & \young(abcd)            & \young(abcd) \\[.04in]\hline

& & \\[-.1in]
$bcda$        & \multirow{2}{*}{$\young(a{\bprime}d,:c)$}    & \multirow{3}{*}{$\young(acd,b)$} \\[.04in]\cline{1-1}
& & \\[-.1in]
$bcad$        &                             &\\[.04in]\cline{1-2}
& & \\[-.1in]
$bacd$        & $\young(a{\bprime}cd)$            & \\[.04in]\hline

& & \\[-.1in]
$acdb$        & \multirow{2}{*}{$\young(abd,:c)$}    & \multirow{3}{*}{$\young(abd,c)$} \\[.04in]\cline{1-1}
& & \\[-.1in]
$acbd$        &                             &\\[.04in]\cline{1-2}
& & \\[-.1in]
$cabd$        & $\young(ab{\cprime}d)$            & \\[.04in]\hline

& & \\[-.1in]
$abdc$        & \multirow{2}{*}{$\young(abc,:d)$}    & \multirow{3}{*}{$\young(abc,d)$} \\[.04in]\cline{1-1}
& & \\[-.1in]
$adbc$        &                             &\\[.04in]\cline{1-2}
& & \\[-.1in]
$dabc$        & $\young(abc{\dprime})$            & \\[.04in]\hline

& & \\[-.1in]
$badc$        & \multirow{2}{*}{$\young(a{\bprime}c,:d)$}    & \multirow{2}{*}{$\young(ac,bd)$} \\[.04in]\cline{1-1}
& & \\[-.1in]
$bdac$        &                            & \\[.04in]\hline

& & \\[-.1in]
$cadb$        & \multirow{2}{*}{$\young(ab{\cprime},:d)$}    & \multirow{2}{*}{$\young(ab,cd)$} \\[.04in]\cline{1-1}
& & \\[-.1in]
$cdab$        &                            & \\[.04in]\hline

& & \\[-.1in]
$cdba$        & \multirow{2}{*}{$\young(a{\bprime}{\cprime},:d)$}    & \multirow{3}{*}{$\young(ad,b,c)$} \\[.04in]\cline{1-1}
& & \\[-.1in]
$cbda$        &                             &\\[.04in]\cline{1-2}
& & \\[-.1in]
$cbad$        & $\young(a{\bprime}{\cprime}d)$            & \\[.04in]\hline

& & \\[-.1in]
$bdca$        & \multirow{2}{*}{$\young(a{\bprime}{\dprime},:c)$}    & \multirow{3}{*}{$\young(ac,b,d)$} \\[.04in]\cline{1-1}
& & \\[-.1in]
$dbca$        &                             &\\[.04in]\cline{1-2}
& & \\[-.1in]
$dbac$        & $\young(a{\bprime}c{\dprime})$            & \\[.04in]\hline

& & \\[-.1in]
$adcb$        & \multirow{2}{*}{$\young(ab{\dprime},:c)$}    & \multirow{3}{*}{$\young(ab,c,d)$} \\[.04in]\cline{1-1}
& & \\[-.1in]
$dacb$        &                             &\\[.04in]\cline{1-2}
& & \\[-.1in]
$dcab$        & $\young(ab{\cprime}{\dprime})$            & \\[.04in]\hline

\raisebox{.8cm}{$dcba$}        & \raisebox{.7cm}{$\young(a{\bprime}{\cprime}{\dprime})$}        & $\young(a,b,c,d)$ \\[.04in]\hline
\end{tabular}
\]

\end{document}